\documentclass[a4paper]
{amsart}

\usepackage{tikz}
\usepackage{subcaption}
\usepackage[dvips]{epsfig}
\usepackage{amscd}
\usepackage{a4wide}
\usepackage{amssymb}
\usepackage{amsthm}
\usepackage{amsmath}
\usepackage{bm} 
\usepackage{latexsym}
\usepackage{enumitem}
\usepackage{tabularx}
\numberwithin{figure}{section}
\numberwithin{table}{section}
\usepackage{graphicx,caption}
\DeclareMathOperator{\sech}{sech}
\DeclareMathOperator{\csch}{csch}
\usepackage{calc}
\usepackage{graphicx}
\usepackage{amsmath,amsthm,bm,mathrsfs}
\usepackage[utf8]{inputenc} % allow utf-8 input
\usepackage{url}            % simple URL typesetting
\usepackage{hyperref}
\hypersetup{
    colorlinks=true,
    linkcolor=blue,
    filecolor=red,      
    urlcolor=blue,
    citecolor=blue,
    }
\usepackage{amsfonts}       % blackboard math symbols
\usepackage{nicefrac}       % compact symbols for 1/2, etc.
\usepackage{microtype}      % microtypography
\theoremstyle{plain}
\usepackage{doi}

\usepackage{upref}
\usepackage{hyperref}
\usepackage{ulem}

\newcommand{\norm}[1]{\left\Vert#1\right\Vert}

\newcommand{\R}{\mathbb R}
\newcommand{\Z}{\mathbb{Z}}

\newcommand{\F}{\mathcal{F}}

\newcommand{\pl}{\partial}

\newtheorem{theorem}{Theorem}[section]
\newtheorem{proposition}[theorem]{Proposition}
\newtheorem{lemma}[theorem]{Lemma}
\newtheorem{definition}[theorem]{Definition}

\newtheorem{remark}[theorem]{Remark}

\numberwithin{equation}{section}     
\numberwithin{figure}{section}
\numberwithin{table}{section}

\allowdisplaybreaks

\newcounter{asnr}

{\ifnum\value{asnr}=0 \stepcounter{asnr} 
  \begin{enumerate}[label=\textbf{A}.\arabic{enumi}]
    \else
    \begin{enumerate}[label=\textbf{A}.\arabic{enumi},resume] \fi}
{\end{enumerate}}

\newcounter{defnr}

{\ifnum\value{defnr}=0 \stepcounter{defnr} 
  \begin{enumerate}[label=\textbf{D}.\arabic{enumi}]
    \else
    \begin{enumerate}[label=\textbf{D}.\arabic{enumi},resume] \fi}
{\end{enumerate}}

\newcommand{\vertiii}[1]{{\left\vert\kern-0.25ex\left\vert\kern-0.25ex\left\vert #1 
\right\vert\kern-0.25ex\right\vert\kern-0.25ex\right\vert}}

\numberwithin{equation}{section} \allowdisplaybreaks

\title[ Convergence OF CNGS for the fractional KdV]
{Stability and Convergence analysis of a Crank-Nicolson Galerkin scheme for the fractional Korteweg-de Vries equation}

\date{}

\author[M. Dwivedi]{Mukul Dwivedi}
\address{
Department of Mathematics, 
	Indian Institute of Technology Jammu,
	Jagti, NH-44 Bypass Road, Post Office Nagrota,
	Jammu - 181221, India}
\email[]{mukul.dwivedi@iitjammu.ac.in}

\author[T. Sarkar]{Tanmay Sarkar}
\address{
	Department of Mathematics, 
	Indian Institute of Technology Jammu,
	Jagti, NH-44 Bypass Road, Post Office Nagrota,
	Jammu - 181221, India}
\email[]{tanmay.sarkar@iitjammu.ac.in}



\thanks{}
\begin{document}
\begin{abstract}
In this paper we study the convergence of a fully discrete Crank-Nicolson Galerkin scheme for the initial value problem associated with the fractional Korteweg-de Vries (KdV) equation, which involves the fractional Laplacian and non-linear convection terms.
Our proof relies on the \emph{Kato type} local smoothing effect to estimate the localized $H^{\alpha/2}$-norm of the approximated solution, where $\alpha \in [1,2)$. We demonstrate that the scheme converges strongly in $L^2(0,T;L^2_{loc}(\mathbb{R}))$ to a weak solution of the fractional KdV equation provided the initial data in $L^2(\mathbb{R})$. Assuming the initial data is sufficiently regular, we obtain the rate of convergence for the numerical scheme. Finally, the theoretical convergence rates are justified numerically through various numerical illustrations.
\end{abstract}

\maketitle

\section{Introduction}\label{sec1}
%%%%%%%% Introduction of the Problem %%%%%%%
This work is concerned with the Cauchy problem associated to a nonlinear, non-local dispersion equation known as a fractional KdV equation
\begin{equation}\label{fkdv}
\begin{cases}
      u_t+\left(\frac{u^2}{2}\right)_x-(-\Delta)^{\alpha/2}u_x=0,  \qquad & (x,t) \in \mathbb{R}\times (0,T],\\
     u(x,0) = u_0(x), \qquad & x \in \mathbb{R},
\end{cases}
\end{equation}
where $T>0$ is fixed, $u_0$ is the prescribed initial condition and $u:\R\times [0,T]\rightarrow \R$ is the unknown.
The non-local operator $-(-\Delta)^{\alpha/2}$ in \eqref{fkdv} is the fractional Laplacian with the values $\alpha\in[1,2)$, defined  for all $\phi\in C_{c}^{\infty}(\mathbb{R}^d)$  by \cite{Hitchhiker}
 \begin{equation}\label{fracL}
     -(-\Delta)^{\alpha/2}[\phi](x) = c_\alpha \text{P.V.}\int_{\mathbb{R}} \frac{\phi(y)-\phi(x)}{|y-x|^{1+\alpha}}\,dy,
 \end{equation}
 for some constant $c_\alpha>0$ which is described in \cite{TanmayS, MonteCarlo}. Moreover, the non-local operator can also be defined through the Fourier transform as
 \begin{equation}\label{fracLF}
     \mathcal{F}[{-(-\Delta)^{\alpha/2}u](\xi)} = |\xi|^\alpha \mathcal{F}[u](\xi),\quad \alpha\in [1,2),
 \end{equation}
 where $\mathcal{F}[u]$ denotes the Fourier transform of $u$. For $\alpha=2$, \eqref{fkdv} reduces to the well-known KdV equation \cite{HOLDEN1999, KATO, Kenig, KdV,Ponce} and whenever $\alpha=1$, \eqref{fkdv} represents the Benjamin-Ono (BO) equation \cite{RajibBO, RDOS, Galtungc,Ponce, TaoBO, Vmurty} which was derived to model the weakly nonlinear internal long waves. In general, the equation \eqref{fkdv} occurs in the study of nonlinear dispersive long waves, inverse scattering method and plasma physics, see \cite{Bonaf, TanmayS, BO_ponce, Kenig, Saut} and references therein. The non-local operators such as fractional Laplacian \eqref{fracL}, have proven to be highly efficient tools in localized computations such as image segmentation, water flow in narrow channels, plasma physics, and other related applications, for more details, refer to \cite{podlubny1998fractional, Hitchhiker} and references therein.

The local and global well-posedness of \eqref{fkdv} have been studied by several authors in recent times. In particular, for $\alpha=2$, well-posedness of \eqref{fkdv} has been studied by several authors in the last three decades, for instance, see \cite{KATO, Ponce, Pilod} and finally global well-posedness in $H^{-1}(\mathbb{R})$ was established by 
Killip et al. \cite{Visan}. Moreover,  in case of $\alpha=1$, the global well-posedness of \eqref{fkdv} in \(H^s(\mathbb{R})\), \(s\geq0\) is proved in \cite{Kenig_wp_BO, TaoBO}. For further discussion on global well-posedness, one can refer to \cite{molinet2012cauchy}. In case of $\alpha\in (1,2)$, the local well-posedness has been established for $H^s(\R),~s\geq 0$ in \cite{Kenig, BO_ponce}. However, the approach of local well-posedness differs significantly from the KdV equation in which contraction principle plays a crucial role. Due to high-low frequency interaction with the nonlinearity, one needs to use the compactness arguments based on a priori estimates on the solution and smoothing effect to establish the local well-posedness of \eqref{fkdv} for $\alpha\in [1,2)$ (see \cite{BO_ponce, VeloS, Vento}). Using the frequency dependent renormalization technique, Herr et al. \cite{FKDV_WP_L2} proved the well-posedness in $L^2(\R)$ for all the values of $\alpha\in (1,2)$. 

Several numerical methods have been developed for \eqref{fkdv} in recent years. In particular, for $\alpha=2$, Sjoberg \cite{SJOBERG} developed a semi-discrete scheme and analyzed the convergence of approximated solution. Due to the requirements of finer grids and reduction in computational cost, many authors have looked for more efficient fully discrete schemes. For instance, Holden et al. \cite{FullydiscreteKDV} designed a fully discrete finite difference scheme and Dutta et al. \cite{CNKDV} proved the convergence of a Crank-Nicolson Galerkin scheme after proposing a higher order scheme in space \cite{dutta2015convergence} for the KdV equation. In a similar manner, for $\alpha=1$, \eqref{fkdv} has been investigated numerically by several authors. A convergent finite difference scheme is developed in \cite{RajibBO} and Galtung \cite{Galtungc, CRGALTUNG} designed a fully discrete Galerkin scheme for the BO equation.
However, there is a limited literature concerning the numerical framework specifically for the fractional KdV equation \eqref{fkdv} with $\alpha\in(1, 2)$. The equation \eqref{fkdv} exhibits two competing effects that contribute to the difficulties encountered in the numerical approximation process. The inclusion of the nonlinear convective term results in the emergence of infinite gradients in finite time, even for smooth initial data. 
Additionally, the presence of the non-local dispersive term produces dispersive waves that are hard to compute with high accuracy and efficiency. Consequently, due to the combined effects of the nonlinear convective term and the dispersive term, designing an accurate and efficient numerical framework for \eqref{fkdv} remains a highly intricate task. Even though there is an operator splitting scheme is studied in \cite{TanmayS}, there is a need for the convergent fully discrete scheme of \eqref{fkdv}.

In this paper, we propose a fully discrete scheme for \eqref{fkdv}. More precisely, the main ingredients of this paper are enlisted below:
\begin{enumerate}
    \item In order to develop an efficient numerical scheme, motivated by the work of \cite{CNGALTUNG, RajibBO} for the BO equation and KdV equation respectively, we look for a higher order approximation in space and second-order fully implicit time steeping scheme. We design a fully discrete Crank-Nicolson type Galerkin scheme for the fractional KdV equation \eqref{fkdv}.
    \item Our goal is to ensure the existence of a sequence of discretized solutions which converge locally to a weak solution of \eqref{fkdv}. This convergence result is established for the low regular initial data \(L^2(\mathbb{R})\). 
    Following the techniques in Kato \cite{KATO}, we make use of intrinsic local smoothing effect of the equation, i.e $u(\cdot,t)\in H^{\alpha/2}(\R)$ $\forall$ $t$. It is worth mentioning that the Kato type smoothing effect exhibited by the fractional KdV equation is stronger than that of the BO equation, but weaker than the KdV equation. This, coupled with the non-local dispersive term, justifies the inclusion of fractional Sobolev spaces, leading to more intricate estimates compared to the case of the KdV equation. 
    \item We investigate the theoretical convergence rates under certain assumptions on the initial data. The obtained rates are further justified by the various numerical illustrations. The real solutions of IVP \eqref{fkdv} posses mainly three conserved quantities:
    \begin{align*}
        C_1(u): & = \int_{\R} u(x,t)~dx, \quad C_2(u):= \int_{\R} u^2(x,t)~dx,\\
         C_3(u):&= \int_{\R} \left((\mathcal{D}^{\alpha/2}u)^2 - \frac{u^3}{3}\right)(x,t)~dx,
    \end{align*}
 where $\mathcal{D}$ is given by $\mathcal{D}:= (-\Delta)^{1/2}$.
A numerical scheme which preserves these conserved quantities are considered to be more accurate. We show that the proposed numerical scheme conserves a discrete version of these quantities.
\end{enumerate}
 
The rest of the paper is organized as follows: In Section \ref{sec2}, we establish some preliminary estimates involving non-local dispersion pertaining to a partially discretized weak formulation of equation \eqref{fkdv}. In addition, we demonstrate the Kato type local smoothing effect. In Section \ref{sec3}, we propose a fully discrete scheme. Since the scheme is implicit in nature, we need to ensure the solvability at each time step. In Section \ref{sec4}, We show the convergence of the scheme to a weak solution of \eqref{fkdv} provided the initial data belongs to $L^2(\R)$. We investigate the theoretical convergence rates in Section \ref{sec5}. Finally, we verify our theoretical findings through several numerical illustrations in Section \ref{sec6}.
%%%%%%%%%%%%%%%%%%%%%%%%%%%%%%%%%%%%%%%%%%%%%%%%%%
\section{Preliminary estimates and local smoothing effect}\label{sec2}
In this section, we shall provide necessary ingredients of our approach. To begin with, we momentarily define a weak solution of the fractional KdV equation \eqref{fkdv} to be a function $u(x,t)\in C^1([0,\infty);H^{1+\alpha}(\mathbb{R}))$ satisfies the following integral formulation:
\begin{equation}\label{WeakF}
\langle u_t,v\rangle + \left\langle\left(\frac{u^2}{2}\right)_x,v\right\rangle + \langle \mathcal{D}^{\alpha}u,v_x\rangle = 0
\end{equation}
for all $v\in H^{1+\alpha}(\mathbb{R})$, where the notation $\langle\cdot,\cdot\rangle$ denotes the usual $L^2$-inner product. 

We plan to discretize the equation \eqref{WeakF} in time using the Crank-Nicolson method. Assuming the time step size to be $\Delta t$, we approximate $u(\cdot,t_n)$ as $u^{n}$, where $t_n = n\Delta t$ and $n$ is a non-negative integer. We introduce the notation: $u^{n+\frac{1}{2}}:= (u^{n+1} + u^n)/2$ and $t_{n+1/2} = (t_n+t_{n+1})/2$. Furthermore, given $u^0$, we define $u^n$ to be the solution of
\begin{equation}\label{CN}
\langle u^{n+1},v\rangle + \Delta t\left\langle\left(\frac{(u^{n+\frac{1}{2}})^2}{2}\right)_x,v\right\rangle + \Delta t\langle \mathcal{D}^{\alpha}u^{n+\frac{1}{2}},v_x\rangle = \langle u^n,v\rangle,
\end{equation}
for any $v\in H^{1+\alpha}(\mathbb{R})$ and $n\geq 0$. Assuming the existence of a unique solution $u^{n+1}$ for the aforementioned equation, we can choose $v=u^{n+1}+u^n$ in \eqref{CN}, leading to
\begin{equation}\label{L2Bound}
\norm{u^{n+1}}^2_{L^2(\mathbb{R})} = \norm{u^n}^2_{L^2(\mathbb{R})} = \norm{u^0}^2_{L^2(\mathbb{R})},
\end{equation}
which further implies
\begin{equation}\label{L2Bound_1}
\norm{u^{n+1/2}}^2_{L^2(\mathbb{R})} \leq  \norm{u^0}^2_{L^2(\mathbb{R})}.
\end{equation}
To derive the above estimates, we have taken into account the following lemma.
\begin{lemma}\label{lemma1}
Let $\mathcal{S}(\R)$ denote the Schwartz space. Then the Fractional Laplacian $ \mathcal{D}^{\alpha}$ with $\alpha\in[1,2)$ on $\mathbb{R}$ defined by \eqref{fracL} and \eqref{fracLF} is a linear operator with the following properties: 
 \begin{enumerate}[label=(\roman*)]
	\item \label{itm2.1}(Symmetric) Assume $f,g \in \mathcal{S}(\mathbb{R})$. Then there holds 
		\[  (\mathcal{D}^{\alpha} f, g) = (f,\mathcal{D}^{\alpha}g).\]
		(An integration by parts formula for fractional Laplacian).
\item \label{itm2.2}(Translation invariant) The Fractional Laplacian commutes with derivatives i.e., for the differentiable functions $f$ and $g$, we have
		\[ (\mathcal{D}^{\alpha}f_x,g) = -(\mathcal{D}^{\alpha}f,g_x) = ((\mathcal{D}^{\alpha}f)_x,g).\]
\item \label{itm2.3}The Fractional Laplacian satisfies
		\[  (\mathcal{D}^{\alpha}f_x,f) = 0.\]
	\item \label{itm2.4}(Semi-group property of fractional Laplacian) Assume $f \in \mathcal{S}(\mathbb{R})$ and $\alpha_1, \alpha_2 >0$, then there holds 
		\[  \mathcal{D}^{\alpha_1}\mathcal{D}^{\alpha_2} f = \mathcal{D}^{\alpha_1+\alpha_2}f.\]
	\item  \label{itm2.5}  $\|f\|^2_{H^{s}(\mathbb{R})} = \|\mathcal{D}^{s}f\|^2_{L^2(\mathbb{R})}$,
		where $s\in(0,1)$.
\end{enumerate}
\end{lemma}
\begin{proof}
	\ref{itm2.1} directly follows from the distributional definition \cite{10Definition} of fractional Laplacian. \ref{itm2.2} will be followed by using the Plancherel theorem, i.e.
	\begin{equation*}
	\begin{split}
	\int_\mathbb{R} \mathcal{D}^\alpha f_x(x)g(x)\,dx &= \int_{\mathbb{R}} \mathcal{F}[\mathcal{D}^\alpha f_x](\xi)\overline{\F[g](\xi)}\,d\xi\\
	& = \int_{\mathbb{R}}-i \xi|\xi|^\alpha \F[f](\xi)\overline{\F[g](\xi)}\,d\xi = \int_\mathbb{R} (\mathcal{D}^\alpha f(x))_xg(x)\,dx\\
	&=\int_{\mathbb{R}}|\xi|^\alpha \F[f](\xi)\overline{i \xi\F[g](\xi)}\,d\xi\\
	&= -\int_{\mathbb{R}} \mathcal{F}[\mathcal{D}^\alpha f](\xi)\overline{\mathcal{F}({g_x})(\xi)}\,d\xi
	= -\int_\mathbb{R} \mathcal{D}^\alpha f(x)g_x(x)\,dx.
	\end{split}
	\end{equation*}
	\ref{itm2.3} is a consequence of the properties \ref{itm2.1} and \ref{itm2.2} as follows:
	\begin{equation*}
	\langle \mathcal{D}^{\alpha}f_x,f\rangle = -\langle f,\mathcal{D}^{\alpha}f_x\rangle \text{  implies  }  \langle \mathcal{D}^{\alpha}f_x,f \rangle =0.
	\end{equation*}
	\ref{itm2.4} can be derived using the Fourier transform as follows: 
	\[ \mathcal{F}[\mathcal{D}^{\alpha_1}\mathcal{D}^{\alpha_2}f](\xi) = |\xi|^{\alpha_1}|\xi|^{\alpha_2}\mathcal{F}[f](\xi) = |\xi|^{\alpha_1+\alpha_2}\mathcal{F}[f](\xi) = \mathcal{F}[\mathcal{D}^{\alpha_1+\alpha_2}f](\xi),\]
	and applying the inverse Fourier transform, we have the required identity. Finally, \ref{itm2.5} is followed from \cite[Proposition 3.6]{Hitchhiker}.
\end{proof}
Furthermore, it can be easily observed that the fractional Laplacian commutes with differentiation, i.e. $\mathcal{D}^\alpha f_x = (\mathcal{D}^\alpha f)_x$. In addition, we mention few results associated to the Sobolev embedding and interpolation inequality which will be instrumental for further analysis.
\begin{lemma}
	Let $p\in[1,\infty)$ and $s,s'>1$ be any two real numbers. Let $\Omega \subseteq \mathbb{R}$ be any open set and $u$ be a measurable function defined on $\Omega$. If $s'\geq s$, then we have 
	\begin{equation*}
	\|u\|_{W^{s,p}}\leq C\|u\|_{W^{s',p}}
	\end{equation*}
	for some positive constant $C$ depending only on $s$ and $p$.
\end{lemma}
\begin{proof}
	For the proof, one can refer to \cite[Proposition 2.1]{Hitchhiker}.
\end{proof} 
\begin{lemma}\label{soboi}\cite[Theorem 6.7]{Hitchhiker}
	Let $s\in(0,1)$ and $p\in[1,\infty)$ such that $sp<n$. Then there exists a positive constant $C= C(n,p,s)$ such that the Sobolev space $W^{s,p}(\mathbb{R}^n)$ is continuously embedded in $L^q(\mathbb{R}^n)$ for any $q\in[p,p^*]$ with $p^*:=np/(n-sp)$, i.e. for any $f\in W^{s,p}(\Omega)$,
	\begin{equation}\label{embed:sobolev}
	\|f\|_{L^q(\mathbb{R}^n)} \leq C \|f\|_{W^{s,p}(\mathbb{R}^n)}.
	\end{equation}
\end{lemma}
\begin{lemma}\cite[Proposition 3.1]{Ponce}
	If $s_1\leq s \leq s_2$ with $s=\theta s_1 + (1-\theta) s_2$, $0\leq \theta \leq 1$, then for any suitable function $u$, we have
	\begin{equation}\label{embed:inter}
	\|u\|_{H^s(\mathbb{R}^n)}\leq \|u\|^{\theta}_{H^{s_1}(\mathbb{R}^n)} \|u\|^{1-\theta}_{H^{s_2}(\mathbb{R}^n)}.
	\end{equation}
\end{lemma}
The above estimates will be essential throughout this paper and for several other identities relevant to fractional Laplacian, one can refer to \cite{10Definition,Hitchhiker}.

In order to derive a local smoothing effect which bounds $u^n$ locally in $H^{\alpha/2}$-norm, we follow the approach originally introduced by Kato \cite{KATO} for the KdV equation. In the work \cite{CNKDV}, it was demonstrated that the solution operator of the KdV equation exhibits a smoothing effect attributed to dispersion. The smoothing effect plays a significant role in the proof of existence of solutions when the initial data belongs to $L^2(\mathbb{R})$. The technique, introduced by Kato \cite{KATO}, is based on the consequence of the commutator identity mentioned in \cite{CNGALTUNG}. We introduce more general commutator operator involving the remainder operator $S_\mu$:
\begin{equation}
[\mathcal{D}^{\mu},h] =: S_\mu, \quad \mu\in \R^{+},
\end{equation}
where $h$ is the operator of multiplication by a smooth function and the commutator bracket $[P,Q]:=PQ-QP$ applies to operators $P$ and $Q$ acting on suitable Sobolev spaces with $0\leq\mu\leq 1$, refer to \cite{VeloC, Galtungc}. Building upon the priori estimate from \cite{biccari,Galtungc,CNGALTUNG}, we have:
\begin{equation*}
\mathcal{D}^{\mu} h = h\mathcal{D}^{\mu} + S_{\mu}(h), \qquad 0\leq\mu \leq 1,
\end{equation*}
and
\begin{equation}
\vertiii{S_\mu} \leq \frac{1}{\sqrt{2\pi}}\left\|\mathcal{F}[{\mathcal{D}^{\mu} h}]\right\|_{L^1(\mathbb{R})},
\end{equation}
where $\vertiii{\cdot}$ denotes the operator norm in $L^2(\mathbb{R})$. 
Hence we have the following estimate:
\begin{equation*}
\begin{split}
\vertiii{S_{\alpha/2}} &\leq \frac{1}{\sqrt{2\pi}}\left\|\mathcal{F}[{\mathcal{D}^{\alpha/2} h}]\right\|_{L^1(\mathbb{R})} = \frac{1}{\sqrt{2\pi}}\left\||\xi|^{\alpha/2}\mathcal{F}[{h}]\right\|_{L^1(\mathbb{R})}\\
& \leq \frac{1}{\sqrt{2\pi}}\left\|(1+|\xi|)\mathcal{F}[{h}]\right\|_{L^1(\mathbb{R})} = \frac{1}{\sqrt{2\pi}}\left(\left\|\mathcal{F}[{h}]\right\|_{L^1(\mathbb{R})} +
\left\|\mathcal{F}[{h}']\right\|_{L^1(\mathbb{R})}\right).
\end{split}
\end{equation*}
%Combining the above results, we have (one can refer to \cite{biccari} for more details) 
%\begin{equation*}
%\mathcal{D}^{\alpha/2}(w\sqrt{\varphi_x}) = \sqrt{\varphi_x}\mathcal{D}^{\alpha/2}w + S_{\alpha/2}(\sqrt{\varphi_x})w.
%\end{equation*}
Note that since $h\in C_c^\infty(\mathbb{R})$, we have $h,h'\in L^1(\mathbb{R})$ and as a consequence, $|\mathcal{F}[{h}'](\xi)|\leq \frac{C}{(1+|\xi|)}$, where $C = C(\|h'\|_{L^1(\mathbb{R})})$. We end up with the following estimate for the remainder operator:
\begin{equation}\label{Remd}
\vertiii{S_{\alpha/2}} \leq \frac{4C}{\sqrt{2\pi}}=:C_S.
\end{equation}
Motivated by Kato \cite{KATO}, we define a smooth cut-off function $\phi$ in the following way
\begin{enumerate}[label=(\roman*)]
	\item \label{item1.1}$1\leq \phi(x) \leq 2+2R,$
	\item $\phi_x(x) = 1 $ for $|x|<R$,
	\item $\phi_x(x) = 0 $ for $|x|\geq R+1$,
	\item $0\leq \phi_x(x) \leq 1,$ for $x\in\mathbb{R}$ and 
	\item \label{item1.5}$\sqrt{\phi_x}\in C_c^\infty(\R)$,
\end{enumerate}
where $R$ denotes a positive constant.
Properties $\ref{item1.1}$-$\ref{item1.5}$ can be obtained by standard mollifier methods (for details, kindly refer to \cite{Galtungc}). Taking into account the properties $\phi$, we consider $v=\phi u^{n+1/2} =: \phi w$ as an admissible test function in $H^{1+\alpha}$. Using the test function $v=\phi w$ in \eqref{CN}, we obtain
\begin{equation}\label{GSl}
\frac{1}{2}\norm{u^{n+1}\sqrt{\phi}}^2_{L^2(\mathbb{R})} + \Delta t \int_\mathbb{R} \mathcal{D}^{\alpha/2}w \mathcal{D}^{\alpha/2}(\phi w)_x\,dx +\Delta t \int_\mathbb{R}\left(\frac{w^2}{2}\right)_xw\phi\,dx = \frac{1}{2}\norm{u^{n}\sqrt{\phi}}^2_{L^2(\mathbb{R})}.
\end{equation}

\begin{proposition}\label{LSE}	Let the initial data $u^0\in L^2(\R)$ and $u^n$ be the solution of \eqref{GSl}. Then we have
	\begin{equation}\label{esti_local_smooth}
	u^{n+1/2}\in \ell^2\left([0,m\Delta t]; H^{\alpha/2}([-R,R])\right),\qquad 0\leq m\leq N,
	\end{equation}
	where $R$ denotes a positive constant.
 %%%%%%%%%%   u^{n+1/2} should be written like u_{\Delta x } 
\end{proposition}

\begin{proof}
We start with the following identity, obtained by using the integration by parts
\begin{equation*}
\int_\mathbb{R}\left(\frac{w^2}{2}\right)_xw\phi\,dx = -\frac{1}{3}\int_\mathbb{R}w^3\phi_x\,dx
\end{equation*}
which can be estimated by applying the Cauchy-Schwartz inequality
\begin{align}
\mathcal{B}_1:= \int_\mathbb{R}w^3\phi_x\,dx \leq& \left(\int_\mathbb{R}w^2\,dx\right)^{1/2}\left(\int_\mathbb{R}w^4\phi^2_x\,dx\right)^{1/2}
= \norm{w}_{L^2(\mathbb{R})}\norm{w\sqrt{\phi_x}}^2_{L^4(\mathbb{R})} \nonumber\\
\leq& C \norm{w}_{L^2(\mathbb{R})}\norm{w\sqrt{\phi_x}}^2_{H^{\alpha/4}(\mathbb{R})}
\leq C \norm{w}_{L^2(\mathbb{R})}\norm{w\sqrt{\phi_x}}_{L^2(\mathbb{R})}\norm{w\sqrt{\phi_x}}_{H^{\alpha/2}(\mathbb{R})} \nonumber\\
\leq& \frac{1}{2} \norm{w\sqrt{\phi_x}}^2_{H^{\alpha/2}(\mathbb{R})} + \frac{C^2}{2} \norm{w}^2_{L^2(\mathbb{R})}\norm{w\sqrt{\phi_x}}^2_{L^2(\mathbb{R})}\nonumber\\
\leq& \frac{1}{2} \|\mathcal{D}^{\alpha/2}(w\sqrt{\phi_x})\|^2_{L^2(\mathbb{R})}+\frac{1}{2}\norm{w\sqrt{\phi_x}}^2_{L^2(\mathbb{R})}+ \frac{C^2}{2} \norm{w}^2_{L^2(\mathbb{R})}\norm{w\sqrt{\phi_x}}^2_{L^2(\mathbb{R})}\nonumber\\
\leq& \frac{1}{2} \|\sqrt{\phi_x}\mathcal{D}^{\alpha/2}w\|^2_{L^2(\mathbb{R})}+\frac{1}{2}\norm{S_{\alpha/2}(\sqrt{\phi_x})w}^2_{L^2(\mathbb{R})} \nonumber \\ 
&\qquad+\frac{1}{2}\norm{w\sqrt{\phi_x}}^2_{L^2(\mathbb{R})}+ \frac{C^2}{2} \norm{w}^2_{L^2(\mathbb{R})}\norm{w\sqrt{\phi_x}}^2_{L^2(\mathbb{R})}\nonumber\\
\leq& \frac{1}{2} \|\sqrt{\phi_x}\mathcal{D}^{\alpha/2}w\|^2_{L^2(\mathbb{R})}+\frac{1}{2}C_S\norm{w}^2_{L^2(\mathbb{R})}+\frac{1}{2}(1+C^2\norm{w}^2_{L^2(\mathbb{R})})\norm{w}^2_{L^2(\mathbb{R})} \label{B_1 esti},
\end{align}
where we have taken into account \eqref{embed:sobolev} and \eqref{embed:inter}. Furthermore, we have the estimate by using the properties of Lemma \ref{lemma1}, $\phi$ and \eqref{Remd} as follows
\begin{align}
\mathcal{B}_2:= & \int_\mathbb{R} \mathcal{D}^{\alpha/2}w \mathcal{D}^{\alpha/2}(\phi w)_x\,dx = \int_\mathbb{R} \mathcal{D}^{\alpha/2}w(\mathcal{D}^{\alpha/2} (\phi w))_x\,dx \nonumber\\
=& \int_\mathbb{R} \mathcal{D}^{\alpha/2}w \big(\phi \mathcal{D}^{\alpha/2}w+S_{\alpha/2}(\phi)w \big)_x\,dx
=\int_\mathbb{R} \mathcal{D}^{\alpha/2}w \Big(\phi_x \mathcal{D}^{\alpha/2}w+\phi \mathcal{D}^{\alpha/2}w_x+(S_{\alpha/2}(\phi)w)_x \Big)\,dx \nonumber\\
=& \int_\mathbb{R} \mathcal{D}^{\alpha/2}w(\phi_x \mathcal{D}^{\alpha/2}w)\,dx+\int_\mathbb{R}\mathcal{D}^{\alpha/2}w(\phi \mathcal{D}^{\alpha/2}w_x)\,dx-\int_\mathbb{R}\mathcal{D}^{\alpha/2}w_xS_{\alpha/2}(\phi)w\,dx \nonumber\\
\geq& \int_\mathbb{R} \mathcal{D}^{\alpha/2}w(\phi_x \mathcal{D}^{\alpha/2}w)\,dx+\int_\mathbb{R}w \mathcal{D}^{\alpha}w_x\,dx-C_S\int_\mathbb{R}w\mathcal{D}^{\alpha/2}w_x \,dx
= \|\sqrt{\phi_x}\mathcal{D}^{\alpha/2}w\|^2 \label{B_2_esti}.
\end{align}
Subsequently, inserting the estimates for $\mathcal{B}_1$ and $\mathcal{B}_2$ in \eqref{GSl}, we obtain
	\begin{equation*}
	\begin{split}
	\frac{1}{2}\norm{u^{n+1}\sqrt{\phi}}^2_{L^2(\mathbb{R})} + \Delta t \|\sqrt{\phi_x}\mathcal{D}^{\alpha/2}w\|^2 
	& \leq \frac{1}{2}\norm{u^{n}\sqrt{\phi}}^2_{L^2(\mathbb{R})} +\frac{\Delta t}{6} \|\sqrt{\phi_x}\mathcal{D}^{\alpha/2}w\|^2_{L^2(\mathbb{R})}\\
	&+\frac{\Delta t}{6}C_S\norm{w}^2_{L^2(\mathbb{R})}+\frac{\Delta t}{6}(1+C^2\norm{w}^2_{L^2(\mathbb{R})})\norm{w}^2_{L^2(\mathbb{R})}
	\end{split}
	\end{equation*}
	which further becomes
	\begin{equation}\label{temp: local_smooth_1}
	\begin{split}
	\frac{1}{2}\norm{u^{n+1}\sqrt{\phi}}^2_{L^2(\mathbb{R})}  + \frac{5\Delta t}{6} \int_{\mathbb{R}}&\phi_x|\mathcal{D}^{\alpha/2}u^{n+1/2}|^2\,dx 
	\leq \frac{1}{2}\norm{u^{n}\sqrt{\phi}}^2_{L^2(\mathbb{R})} + C\left(\norm{u^0}_{L^2(\mathbb{R})}\right)\Delta t,
	\end{split}
	\end{equation}
where the final estimate in \eqref{temp: local_smooth_1} is due to the boundedness of the $L^2$-norm of $w$ (cf. \eqref{L2Bound_1}). After dropping the positive second term on the left hand side in \eqref{temp: local_smooth_1} and subsequently, summing over $n=0$ to $n=m-1$ and observing that this is a telescoping sum, we end up with
\begin{equation}\label{temp: local_smooth_2}
	\|u^m\sqrt{\phi}\|^2_{L^2(\mathbb{R})}\leq \|u^0\sqrt{\phi}\|^2_{L^2(\mathbb{R})} + C\left(\norm{u^0}_{L^2(\mathbb{R})}\right)m\Delta t.
\end{equation}
Again in \eqref{temp: local_smooth_1}, first taking sum over from $n=0$ to $n=m$ and then dropping the term
$\frac{1}{2}\|u^{k+1}\sqrt{\phi}\|^2_{L^2(\mathbb{R})}$ from the resulting estimate yields 
	\begin{equation}\label{temp: local_smooth_3}
	\Delta t \sum_{n=0}^m\int_{-R}^{R}\phi_x|\mathcal{D}^{\alpha/2}u^{n+1/2}|^2\,dx
	\leq \frac{6}{5}\left(\frac{1}{2}\norm{u^{n}\sqrt{\phi}}^2_{L^2(\mathbb{R})} + C\left(\norm{u^0}_{L^2(\mathbb{R})}\right)(m+1)\Delta t\right).
	\end{equation}
	Combining the above estimates \eqref{temp: local_smooth_2} and \eqref{temp: local_smooth_3}, we obtain \eqref{esti_local_smooth}. Hence the result follows.
\end{proof}

\begin{remark}
 The Kato type local smoothing effects in Proposition \ref{LSE} illustrates that the temporal discretization of the sequence space mentioned in \eqref{esti_local_smooth} serves as a discrete analogous representation of $L^2([0,T);H^{\alpha/2}([-R,R]))$. We would like to point out that the solution obtained from the Crank-Nicolson temporal discretized equation demonstrate a similar local smoothing effect observed in the BO equation \cite{CNGALTUNG} and the KdV equation \cite{dutta2015convergence}.
\end{remark}
\begin{remark}
 The local smoothing effect will play a crucial role in demonstrating the convergence of the discretized solution of \eqref{fkdv}. The idea behind our approach is to consider the test function of the form $\phi v$, where $v$ belongs to a suitable finite element space. This approach was implemented for KdV equation in \cite{dutta2015convergence} and for BO equation in \cite{Galtungc}. The choice of test function has an advantage to obtain $H^{\alpha/2}$-bound.
 However, the $L^2$-estimate in \eqref{L2Bound} no longer follows directly by choosing an appropriate test function.
\end{remark}
Our objective is to derive the $L^2$-estimate for the approximate solution of \eqref{fkdv} while using the test function of the type $\phi v$. In order to use the local smoothing effect, we introduce a suitable CFL condition which will be instrumental to find the $L^2$-bounds.
%%%%%%%%%%%%%%
 \section{Discrete formulation}\label{sec3}
In this section, we propose a Galerkin scheme for spatial discretization along with the Crank-Nicolson scheme for temporal discretization. At each discrete time step, our aim is to define a sequence of functions approximating the exact solution of \eqref{fkdv} using the weak formulation. Furthermore, an iteration scheme will be devised to solve the implicit equation for each time step. As a consequence, we also need to demonstrate the solvability of the resulting equation. We start by introducing the following notations.
\subsection{Notation}
To establish the finite element space, we partition the spatial domain into equally-sized elements in the form of intervals. For each $j\in\mathbb{Z}$, let the grid points be $x_j = j\Delta x$, where $\Delta x$ represents the spatial step length. Consequently, we denote the spatial grid cells $I_j = [x_{j-1},x_j]$. Similarly, for the discretization of time, we fix a time horizon $T>0$ and set $t_n = n\Delta t$ for $n\in \{0,1,2,.......,N\}$, where $\Delta t$ is the temporal step length and $\left(N+\frac{1}{2}\right)\Delta t =T.$
Furthermore, we introduce the notation $t_{n+1/2} = (t_n+t_{n+1})/2$.
 
 We define the weighted $L^2$-inner product with the weight function $\phi$ defined in the previous section
 \begin{equation*}
     \langle u,v\rangle_\phi := \langle u,v\phi\rangle,
 \end{equation*}
 and the associated weighted norm $\|u\|_{2,\phi} = \sqrt{\langle u,u\phi\rangle}$.
\subsection{Galerkin Scheme}
We propose a Galerkin scheme for \eqref{fkdv} using the weak formulation \eqref{WeakF}. In particular, we seek an approximation $u_{\Delta x}$ of \eqref{fkdv} belongs to the finite element space
\begin{equation*}
	S_{\Delta x} = \{v\in H^{1+\alpha}(\mathbb{R}) ~| ~v\in \mathbb{P}_r(I_j),~j\in \Z\}
\end{equation*}
for all $t\in [0,T]$, where $r\geq2$, and $\mathbb{P}_r(I_j)$ denotes the space of polynomials on $I_j$ with degree at most $r$.

We define $\mathcal{P}$ as the $L^2$-orthogonal projection onto the space $S_{\Delta x}$ and set $u^0 = \mathcal{P}u_0$. We define the sequence $\{u^n\}$ such that $u^{n+1}\in S_{\Delta x}$ satisfies
\begin{equation}\label{CNscheme}
     \langle u^{n+1},\phi v\rangle - \Delta t \left\langle\frac{(u^{n+1/2})^2}{2},(\phi v)_x\right\rangle + \Delta t\left\langle \mathcal{D}^{\alpha/2} u^{n+1/2} ,\mathcal{D}^{\alpha/2} (\phi v)_x\right\rangle = \langle u^n,\phi v\rangle,
 \end{equation}
for all $v\in S_{\Delta x}$ and $n\in \{0,1,...,N\}$. Due to the implicit nature of the scheme \eqref{CNscheme} and in order to find $u^{n+1}$ from the nonlinear equation \eqref{CNscheme}, we need to consider the solvability at each time step. 
Since $\mathcal{P}$ is the $L^2$-projection operator, there holds $\norm{u^0}_{L^2(\mathbb{R})}\leq \norm{u_0}_{L^2(\mathbb{R})}$.
\subsection{Solvability for each time step}
In order to demonstrate the existence of a solution $u^n$ for each time step, we refer to the approach in \cite{CNKDV} and \cite{Galtungc} for KdV equation and BO equation respectively. Let us define the iterative scheme as follows: for every $v\in S_{\Delta x}$, 
\begin{equation}\label{solv}
\begin{cases}
    &\langle w^{\ell+1},\phi v\rangle - \displaystyle\frac{\Delta t}{2}\left\langle\left(\frac{w^\ell+u^n}{2}\right)^2,(\phi v)_x\right\rangle+\Delta t\left\langle \mathcal{D}^{\alpha/2}\left( \frac{w^{\ell+1}+u^n}{2}\right), \mathcal{D}^{\alpha/2}(\phi v)_x\right\rangle = \langle u^n, \phi v\rangle,\\
    &w^0 = u^n.
    \end{cases}
\end{equation}
The above iterative scheme can be considered as a Galerkin scheme of a linear problem involving the bilinear form in $w^{\ell+1}$ and $u^n$. The scheme \eqref{solv} can be rewritten as 
\begin{equation}\label{solv2}
    \left\langle \left(1-\frac{\Delta t}{2}\mathcal{D}^\alpha_x \right)w^{\ell+1},\phi v\right\rangle  = \displaystyle\frac{\Delta t}{2}\left\langle\left(\frac{w^\ell+u^n}{2}\right)^2,(\phi v)_x\right\rangle+\frac{\Delta t}{2}\left\langle \mathcal{D}_x^{\alpha}u^n, \phi v \right\rangle + \langle u^n, \phi v\rangle.
\end{equation}
Let $d_{ij} = \left\langle \mathcal{D}^\alpha_x p_j,p_i \right\rangle = - \left\langle \mathcal{D}^\alpha_x p_i,p_j \right\rangle = - d_{ji}$, where $\{p_j\}$ form an orthogonal basis for $S_{\Delta x}$ and $\mathcal{D}^\alpha_x p_j = \mathcal{D}^\alpha (p_j)_x$ and we are choosing $v = \frac{p_i}{\phi}$ in \eqref{solv2}. Then
the matrix on the left hand side of \eqref{solv2}, say $I-\frac{\Delta t}{2}D$, is non-singular for sufficiently small $\Delta t$, where $D=(d_{ij})$ is skew-symmetric. Hence the existence of $w^{l+1}$ in \eqref{solv} is ensured. Afterwards, the solvability of the implicit scheme \eqref{CNscheme} is guaranteed by the following lemma.
%%%%%%%%%Lemma%%%%%%%%
\begin{lemma}\label{Lemma:solvability}
Let us consider the iterative scheme \eqref{solv} and assume that the following CFL condition holds:
\begin{equation}\label{cfl}
	\lambda\leq\frac{L}{2\sqrt{C_2}K\norm{u^n}_{2,\phi}},
\end{equation}
where the constant $L$ is chosen such that $0<L<1$ and $K$ is defined by
\begin{equation*}
    K = \frac{5-L}{1-L}>5.
\end{equation*}
 The constant $C_2$ in \eqref{cfl} is independent of $u^n$ and $\Delta x$, and $\lambda$ is given by $$\lambda = \frac{\Delta t}{\Delta x^{3/2}}.$$ Then there exists a function $u^{n+1}$ that solves \eqref{CNscheme}, and $\lim_{\ell\xrightarrow[]{}\infty}w^\ell= u^{n+1}$. Furthermore, the following estimate holds
\begin{equation}
    \norm{u^{n+1}}_{2,\phi} \leq K\norm{u^n}_{2,\phi}.
\end{equation}
\end{lemma}
%%%%%%%%%%%%%%%%%%%%%%%%%
\begin{proof}
We introduce the notation
 \begin{equation*}
	\mathcal{G}(u^n, \phi v) := \langle u^n, \phi v\rangle +\frac{\Delta t}{8}\left\langle(u^n)^2, (\phi v)_x\right\rangle- \frac{\Delta t}{2}\left\langle \mathcal{D}^{\alpha/2} u^n, \mathcal{D}^{\alpha/2}(\phi v)_x\right\rangle,
\end{equation*}
and consequently, the iterative scheme \eqref{solv} can be rewritten as 
 \begin{equation*}
     \langle w^{\ell+1},\phi v\rangle + \frac{\Delta t}{4}\left\langle(w^\ell u^n)_x,\phi v\right\rangle+\frac{\Delta t}{4}\left\langle w^\ell w^\ell_x,\phi v\right\rangle+\frac{\Delta t}{2}\left\langle \mathcal{D}^{\alpha/2} w^{\ell+1}, \mathcal{D}^{\alpha/2} (\phi v)_x\right\rangle = \mathcal{G}(u^n, \phi v),
 \end{equation*}
 for all $v\in S_{\Delta x}$. Since the parameter $\ell$ is not involved in the right hand side, we obtain
 \begin{equation*}
   \begin{split}
     \langle w^{\ell+1}-w^\ell,\phi v\rangle + \frac{\Delta t}{4}\left\langle \big(u^n(w^\ell - w^{\ell-1})\big)_x,\phi v\right\rangle &+\frac{\Delta t}{4}\left\langle w^\ell w^\ell_x -w^{\ell-1} w^{\ell-1}_x ,\phi v\right\rangle \\
     &+\frac{\Delta t}{2}\left\langle \mathcal{D}^{\alpha/2} (w^{\ell+1}-w^{\ell}), \mathcal{D}^{\alpha/2}(\phi v)_x\right\rangle = 0.
    \end{split}
 \end{equation*}
We choose $v=w^{\ell+1}-w^\ell =:w$ in the above equation to obtain
 \begin{equation}\label{temp:sol_0}
   \begin{split}
     \langle w,\phi w\rangle +\frac{\Delta t}{2} & \left\langle \mathcal{D}^{\alpha/2} w, \mathcal{D}^{\alpha/2}(\phi w)_x\right\rangle \\& = \underbrace{-\frac{\Delta t}{4}\left\langle(u^n(w^\ell - w^{\ell-1}))_x,\phi w\right\rangle}_{\mathcal{A}_1} \underbrace{-\frac{\Delta t}{4}\left\langle w^\ell w^\ell_x -w^{\ell-1} w^{\ell-1}_x ,\phi w\right\rangle}_{\mathcal{A}_2}.
    \end{split}
 \end{equation}
 Taking into account the estimate of the fractional Laplacian in Proposition \ref{LSE} and using the fact that $\phi\geq 1$, we have
 \begin{equation}\label{temp:sol_1}
     \left\langle \mathcal{D}^{\alpha/2} w, \mathcal{D}^{\alpha/2}(\phi w)_x\right\rangle \geq \|\sqrt{\phi_x}\mathcal{D}^{\alpha/2}w\|^2_{L^2(\R)} \geq 0.
 \end{equation}
 Afterwards, we estimate $\mathcal{A}_2$ using the Young's inequality
 \begin{equation*}
     \begin{split}
         \mathcal{A}_2 &= \frac{1}{4}\int_\mathbb{R} (-\Delta t)(w^\ell w^\ell_x -w^{\ell-1} w^{\ell-1}_x )\phi w \,dx\\
         %&\leq \frac{\Delta t^2}{8}\int_\mathbb{R} (w^\ell w^\ell_x -w^{\ell-1} w^{\ell-1}_x )^2\phi \, dx + \frac{1}{8}\int_\mathbb{R} w^2\phi \,dx\\
         &\leq \frac{\Delta t^2}{8}\int_\mathbb{R} ((w^\ell-w^{\ell-1}) w^\ell_x + w^{\ell-1} (w^\ell_x-w^{\ell-1}_x ))^2\phi \, dx + \frac{1}{8}\int_\mathbb{R} w^2\phi \,dx\\
         %&\leq \frac{\Delta t^2}{4}\int_\mathbb{R} (w^\ell-w^{\ell-1})^2 (w^\ell_x)^2 \phi\,dx+ \frac{\Delta t^2}{4}\int_\mathbb{R}(w^{\ell-1})^2 (w^\ell_x-w^{\ell-1}_x )^2\phi \, dx + \frac{1}{8}\int_\mathbb{R} w^2\phi \,dx\\
         &\leq \frac{\Delta t^2}{4}\left(\|w^\ell_x\|^2_{L^\infty(\mathbb{R})}\|w^\ell-w^{\ell-1}\|^2_{2,\phi} + \|w^\ell_x-w^{\ell-1}_x \|^2_{L^\infty(\mathbb{R})}\|w^{\ell-1}\|^2_{2,\phi}\right)+ \frac{1}{8}\norm{w}^2_{2,\phi}\\
         &\leq \frac{C_2\Delta t^2}{4\Delta x^3}\left(\norm{w^\ell}^2_{L^2(\mathbb{R})}\|w^\ell-w^{\ell-1}\|^2_{2,\phi} + \|w^\ell-w^{\ell-1}\|^2_{L^2(\mathbb{R})}\|w^{\ell-1}\|^2_{2,\phi}\right)+ \frac{1}{8}\norm{w}^2_{2,\phi},
     \end{split}
 \end{equation*}
 where we have used the following inverse inequalities \cite[p. 142]{Ciarlet} in the last estimate
 \begin{align}\label{inv_inequality}
 	\norm{w^{\ell}_x}_{L^{\infty}(\R)} \leq \frac{C_1^{1/2}}{(\Delta x)^{1/2}} \norm{w^{\ell}_x}_{L^2(\R)}
 	\leq \frac{C_2^{1/2}}{(\Delta x)^{3/2}} \norm{w^{\ell}}_{L^2(\R)},
 \end{align}
 where $C_1$ and $C_2$ are independent of $w^{\ell}$ and $\Delta x$.
 Using the definition of $\lambda$, we end up with
 \begin{equation}\label{temp:sol_2}
     \mathcal{A}_2 \leq \frac{1}{8}\norm{w}^2_{2,\phi} +\frac{C_2\lambda^2}{2}
     \max\{ \norm{w^\ell}^2_{2,\phi},\|w^{\ell-1}\|^2_{2,\phi}\}\|w^\ell-w^{\ell-1}\|^2_{2,\phi}.
 \end{equation}
 In a similar way, we also get the estimate of $\mathcal{A}_1$
  \begin{equation*}
     \mathcal{A}_1 \leq \frac{\Delta t^2}{8} \int_\mathbb{R} \big((u^n(w^\ell-w^{\ell-1}))_x \big)^2 \phi\,dx + \int_\mathbb{R} w^2\phi \,dx
 \end{equation*}
 which in turn becomes
 \begin{equation}\label{temp:sol_3}
     \mathcal{A}_1 \leq \frac{1}{8}\norm{w}^2_{2,\phi} +\frac{C_2\lambda^2}{2}\norm{u^n}^2_{2,\phi}\|w^\ell-w^{\ell-1}\|^2_{2,\phi}.
 \end{equation}
 Combining the estimates \eqref{temp:sol_1}, \eqref{temp:sol_2} and \eqref{temp:sol_3} together in \eqref{temp:sol_0}, we have the following bound for $\ell\geq 1$,
 \begin{equation}\label{wll}
     \norm{w^{\ell+1}-w^\ell}^2_{2,\phi} \leq  \frac{2}{3}C_2\lambda^2\max\{ \norm{w^\ell}^2_{2,\phi},\|w^{\ell-1}\|^2_{2,\phi},\norm{u^n}^2_{2,\phi}\}\|w^\ell-w^{\ell-1}\|^2_{2,\phi}.
 \end{equation}
 For the bound of $w^1$, setting $\ell = 0$ in \eqref{solv}, we get
 \begin{equation*}
    \begin{split}
     \langle w^{1} - u^n,\phi v\rangle +\Delta t\left\langle \mathcal{D}^{\alpha}\left( \frac{w^{1}+u^n}{2}\right), (\phi v)_x\right\rangle = \frac{\Delta t}{2}\left\langle\left(u^n\right)^2,(\phi v)_x\right\rangle 
     &= -\Delta t\left\langle u^n u^n_x, \phi v \right\rangle.
     \end{split}
 \end{equation*}
 Choosing $v=\frac{u^n+w^1}{2}$ yields
\begin{equation*}
     \frac{1}{2}\int_{\mathbb{R}} \left((w^{1})^2 - (u^n)^2\right)\phi \,dx +\Delta t\int_{\mathbb{R}} \mathcal{D}^{\alpha}\left( \frac{w^{1}+u^n}{2}\right)\left(\phi \frac{u^n+w^1}{2}\right)_x \,dx
     = -\Delta t\int_{\mathbb{R}} u^n u^n_x \phi \frac{u^n+w^1}{2}\,dx .
 \end{equation*}
 Again, estimating the term involving fractional Laplacian as before, using the Young's inequality and inverse inequality leads to 
 \begin{equation}\label{w1}
     \|w^1\|^2_{2,\phi} \leq 4(1+C_2\lambda^2\norm{u^n}^2_{2,\phi})\norm{u^n}^2_{2,\phi}.
 \end{equation}
 Then we claim that the following holds
 %\begin{equation}
     %\begin{split}
         %\norm{w^{\ell+1}-w^\ell}_{2,\phi} &\leq L \|w^{\ell}-w^{\ell-%1}\|_{2,\phi},\label{a}\\
    %     \norm{w^\ell}_{2,\phi} &\leq K \norm{u^n}_{2,\phi},\\
   %      \|w^1\|_{2,\phi} &\leq K \norm{u^n}_{2,\phi},
  %   \end{split}
 %\end{equation}
 \begin{eqnarray}\label{ids}
          \label{1}\norm{w^{\ell+1}-w^\ell}_{2,\phi}&\leq&L\|w^{\ell}-w^{\ell-1}\|_{2,\phi},\label{a}, \quad \ell\geq 1,\\
         \label{2}\norm{w^\ell}_{2,\phi} &\leq&K\norm{u^n}_{2,\phi}, \quad \ell\geq 1,\\
         \label{3}\|w^1\|_{2,\phi} &\leq& 3 \norm{u^n}_{2,\phi}.
 \end{eqnarray}
 We prove the claim by an induction argument. We use the CFL condition \eqref{cfl}
 and \eqref{w1} to get
 \begin{equation*}
     \begin{split}
         \|w^1\|_{2,\phi} &\leq \left(2+2\sqrt{C_2}\lambda\norm{u^n}_{2,\phi}\right)\norm{u^n}_{2,\phi}\\
         &\leq \left(2+\frac{L}{K}\right)\norm{u^n}_{2,\phi}\leq 3 \norm{u^n}_{2,\phi} \leq K \norm{u^n}_{2,\phi},
     \end{split}
 \end{equation*}
 and hence \eqref{2} holds for $\ell =1$. Setting $\ell =1 $ in \eqref{wll} and using \eqref{cfl}, we obtain
 \begin{equation*}
     \begin{split}
     \|w^{2}-w^1\|_{2,\phi} &\leq  \sqrt{\frac{2}{3}C_2}\lambda\max\{ \|w^1\|_{2,\phi},\norm{u^n}_{2,\phi}\}\|w^1-u^n\|_{2,\phi}\\
     & \leq  \left(3\sqrt{\frac{2}{3}C_2}\lambda \norm{u^n}_{2,\phi}\right)\|w^1-u^n\|_{2,\phi}\\
     & \leq \frac{3L}{2K}\|w^1-u^n\|_{2,\phi}\leq L \|w^1-u^n\|_{2,\phi}.
     \end{split}
 \end{equation*}
  Hence \eqref{1} holds for $\ell =1$. Afterwards, we assume that \eqref{1} and \eqref{2}  hold for $\ell =1,2,3,...,m$. Then
 \begin{equation*}
 \begin{split}
     \|w^{m+1}\|_{2,\phi} &\leq \sum_{\ell=0}^m \norm{w^{\ell+1}-w^\ell}_{2,\phi} + \|w^0\|_{2,\phi}\leq \|w^1-w^0\|_{2,\phi}\sum_{\ell=0}^mL^\ell + \|w^0\|_{2,\phi}\\
     &\leq 4\norm{u^n}_{2,\phi}\frac{1}{1-L} + \norm{u^n}_{2,\phi} = \frac{5-L}{1-L}\norm{u^n}_{2,\phi} =K\norm{u^n}_{2,\phi}.
  \end{split}   
 \end{equation*}
 Thus \eqref{2} holds for all $\ell$. The above estimate together with \eqref{wll} and \eqref{cfl} leads to
 \begin{equation*}
     \begin{split}
          \|w^{m+2}-w^{m+1}\|_{2,\phi} &\leq  \sqrt{\frac{2}{3}C_2}\lambda\max\{ \|w^{m+1}\|_{2,\phi},\|w^{m}\|_{2,\phi},\norm{u^n}_{2,\phi}\}\|w^{m+1}-w^{m}\|_{2,\phi}\\
          &\leq  \sqrt{\frac{2}{3}C_2} \lambda K\norm{u^n}_{2,\phi}\|w^{m+1}-w^{m}\|_{2,\phi} \leq L \|w^{m+1}-w^{m}\|_{2,\phi}.
     \end{split}
 \end{equation*}
 This shows that \eqref{1} holds for all $\ell$ as well. Since we have $0<L<1$, the sequence $\{w^\ell\}$ is a Cauchy sequence and it converges to $u^{n+1}$. Hence the proof follows.
\end{proof}
%%%%%%%%%%%%%%%%%%%%%%%%%%%%%%%%%%%%%%%%
\section{Convergence of the scheme}\label{sec4}
In this section, we will prove the convergence of the proposed Crank-Nicolson Galerkin scheme \eqref{CNscheme}.  As we have mentioned before, the inherent Kato type local smoothing effect of the fractional KdV equation \eqref{fkdv} helps us to obtain the $H^{\alpha/2}_{\text{loc}}(\mathbb{R})$-estimate of the approximate solution induced by the scheme \eqref{CNscheme}.

\begin{lemma}\label{Lemma:H_alpha_bnd}
    Let $\lambda,$ $K$ and $L$ be defined as in Lemma \ref{Lemma:solvability}  and $u^n$ be the solution of the scheme \eqref{CNscheme}. Assume the time step $\Delta t$ satisfies 
    \begin{equation}\label{cfl2}
        \lambda\leq\frac{L}{2\sqrt{C_2}K\sqrt{Y}},
    \end{equation}
    where $Y=Y\big(\norm{u_0}_{L^2(\mathbb{R})} \big)$. 
    Then there exist time $T>0$ and a constant $C$, both depending only on $\norm{u_0}_{L^2(\mathbb{R})}$ such that for all $n$ satisfying 
    $n\Delta t\leq T$, the following estimate holds
    \begin{equation}\label{Stab}
        \norm{u^n}_{L^2(\mathbb{R})}\leq C\left(\norm{u_0}_{L^2(\mathbb{R})}\right).
    \end{equation}
    In addition, the approximation $u^n$ satisfies the following $H^{\alpha/2}$-estimate
    \begin{equation}\label{Dalp}
        \Delta t\sum_{\left(n+\frac{1}{2}\right)\Delta t\leq T}\left\|\mathcal{D}^{\alpha/2}u^{n+1/2}\right\|^2_{L^2([-R,R])}\leq C\left(\norm{u_0}_{L^2(\mathbb{R})}\right).
    \end{equation}
\end{lemma}
\begin{proof}
 We observe that  \eqref{GSl} holds by choosing $v= u^{n+1/2}$ in  \eqref{CNscheme} and consequently, using the estimates in \eqref{B_1 esti} and \eqref{B_2_esti}, we deduce
    \begin{equation*}
        \begin{split}
            \int_{\mathbb{R}}&\left(u^{n+1}\right)^2\phi\,dx + 2\Delta t\int_{\mathbb{R}}\left|\mathcal{D}^{\alpha/2}u^{n+1/2}\right|^2\phi_x\,dx\\
            &\leq\int_{\mathbb{R}}\left(u^{n}\right)^2\phi\,dx + \frac{2\Delta t}{3}\int_{\mathbb{R}}\left|\mathcal{D}^{\alpha/2}u^{n+1/2}\right|^2\phi_x\,dx+\frac{2\Delta t}{3} C_s \int_{\mathbb{R}}\left(u^{n+1/2}\right)^2\phi\,dx \\
            &\qquad\qquad+ \frac{\Delta t}{3} \int_{\mathbb{R}}\left(u^{n+1/2}\right)^2\phi\,dx
            + \frac{\Delta t}{3} C^2\left(\int_{\mathbb{R}}\left(u^{n+1/2}\right)^2\phi\,dx \right)^2
        \end{split}
    \end{equation*}
    which further turns into
     \begin{equation}\label{eq25}
        \begin{split}
            \int_{\mathbb{R}}&\left(u^{n+1}\right)^2\phi\,dx + \frac{4}{3}\Delta t\int_{\mathbb{R}}\left|\mathcal{D}^{\alpha/2}u^{n+1/2}\right|^2\phi_x\,dx\\
            &\leq\int_{\mathbb{R}}\left(u^{n}\right)^2\phi\,dx + C \Delta t \left[\int_{\mathbb{R}}\left(u^{n+1/2}\right)^2\phi\,dx +\left(\int_{\mathbb{R}}\left(u^{n+1/2}\right)^2\phi\,dx \right)^2\right].
        \end{split}
    \end{equation}
 By dropping the non-negative term involving the fractional derivative in \eqref{eq25} and denoting $a_n:=\int_{\mathbb{R}}\left(u^n\right)^2\phi\,dx$ implies 
    \begin{equation}\label{an}
        a_{n+1} \leq a_n + \Delta t f(a_{n+\frac{1}{2}}),
    \end{equation}
  where the function $f$ is given by
    \begin{equation*}
        f(a) = C (a+a^2).
    \end{equation*}
    It can be easily seen that $a_{n+\frac{1}{2}}\leq (a_n+a_{n+1})/2$ and hence $\{a_n\}$ satisfies the differential inequality
    \begin{equation*}
        \frac{da}{dt}\leq f(a).
    \end{equation*}
    We claim that 
    \begin{align}\label{claim_1}
    	a_n\leq y(t_n) \quad \text{for every } n\geq 0,
    \end{align}
    where $y$ satisfies the following initial value problem
    \begin{equation}\label{temp:ode}
        \begin{cases}
            \frac{dy}{dt} = f\left(\frac{K^2+1}{2}y\right),\qquad t>0,\\
            y(0) = a_0,
        \end{cases}
    \end{equation}
 in which $K$ occurs from the Lemma \ref{Lemma:solvability}. Thanks to the locally Lipschitz contintinuity of $f$,  the solution of \eqref{temp:ode} is unique and moreover, it is strictly increasing and convex. However, it blows up in finite time, say at $t = T_\infty$ and hence we choose $T=T_{\infty}/2$. We argue via induction to prove \eqref{claim_1}. We assume that \eqref{claim_1} holds for $n\in\{0,1,2,...,m\}$. As $0\leq a_m\leq y(T)$, \eqref{cfl2} implies that \eqref{cfl} holds, and thus Lemma \ref{Lemma:solvability} gives $a_{m+1}\leq K^2a_m$. As a consequence, using the convexity of $f$, we obtain
 \begin{equation*}
 	\begin{split}
 		a_{m+1}&\leq a_m+\Delta tf\left(\frac{K^2+1}{2}a_m\right) \leq y(t_m) +\Delta tf\left(\frac{K^2+1}{2}y(t_m)\right) 
 		\leq y(t_{m+1}).
 	\end{split}
 \end{equation*}
 Hence the claim \eqref{claim_1} follows. 
 Since $\phi\geq 1$, we obtain the $L^2$-stability estimate \eqref{Stab}.
 Afterwards, summing \eqref{eq25} over $n$ yields the estimate
    \begin{equation*}
        \Delta t\sum_{\left(n+\frac{1}{2}\right)\Delta t\leq T}\left\|\mathcal{D}^{\alpha/2}u^{n+1/2}\right\|^2_{L^2([-R,R])}\leq C\left(\norm{u_0}_{L^2(\mathbb{R})}\right).
    \end{equation*}
    This completes the proof.
\end{proof}

\subsection{Bounds on temporal derivatives} 
In order to carry out the convergence analysis, we require certain temporal derivative bounds on the approximated solution of \eqref{CNscheme}. For this, we need to introduce a suitable projection operator. By adopting the approach from the monograph of Ciarlet \cite{Ciarlet}, we can ensure the existence of the projection map $P$ on $S_{\Delta x}$. More precisely, we have 
\begin{lemma}\label{Lemma:projection}
    Let $\psi\in C_c^\infty(-R,R)$ and $\phi$ be defined by properties $(i)-(iv)$ in Section \ref{sec2}. Then there exists a projection operator $P:C_c^\infty(-R,R)\xrightarrow[]{} S_{\Delta x}\cap C_c(-R,R)$ such that 
    \begin{equation}\label{projprop}
        \int_{\mathbb{R}}uP(\psi)\phi\,dx  = \int_{\mathbb{R}}u\psi\phi\,dx, \qquad u\in S_{\Delta x}.
    \end{equation}
    In addition, $P$ satisfies the bounds 
    \begin{eqnarray}\label{Proj}
        \begin{cases}
            \label{P1}\|P(\psi)\|_{L^2(\mathbb{R})}\leq C\|\psi\|_{L^2(\mathbb{R})},\\
            \label{P2}\|P(\psi)\|_{H^1(\mathbb{R})}\leq C\|\psi\|_{H^1(\mathbb{R})},\\
            \label{P3}\|P(\psi)\|_{H^2(\mathbb{R})}\leq C\|\psi\|_{H^2(\mathbb{R})},
        \end{cases}
    \end{eqnarray}
    where the constant $C$ is independent of $\Delta x$.\\
    Moreover, there holds
    \begin{equation}\label{fracP}
        \|P(\psi)\|_{H^{\alpha/2}(\mathbb{R})}\leq C\|\psi\|_{H^{\alpha/2}(\mathbb{R})}.
    \end{equation}
\end{lemma}
For the proof of \eqref{Proj}-\eqref{fracP}, one can refer to \cite{Galtungc, StabOnProj}.
From the definition of the dual norms in $H^{-s}(\mathbb{R})$ and $H^{-s}([-R,R])$  for all $s\in\mathbb{R}^+$, we have the estimates 
\begin{equation}\label{siq1}
	\int_{\mathbb{R}}  fg\,dx \leq \|f\|_{H^{-s}(\mathbb{R})}\|g\|_{H^{s}(\mathbb{R})} \quad\text{for all } f\in H^{-s}(\mathbb{R}), g\in H^{s}(\mathbb{R}),
\end{equation}
 and for $f\in H^{-s}([-R,R]), g\in H^{s}([-R,R])$, there holds
\begin{equation}\label{siq2}
	\int_{-R}^{R} fg\,dx \leq \|f\|_{H^{-s}([-R,R])}\|g\|_{H^{s}([-R,R])}.
\end{equation}
With the help of Lemma  \ref{Lemma:projection}, we have the estimates on the temporal derivatives stated as follows:
\begin{lemma}
    Let $\{u^n\}$ be the solution of the scheme \eqref{CNscheme}. Suppose the hypothesis of Lemma \ref{Lemma:H_alpha_bnd} holds. Then we have the estimate
    \begin{equation}\label{TempDB}
        \|D^+_tu^n\phi\|_{H^{-2}([-R,R])}\leq C \left(\norm{u_0}_{L^2(\mathbb{R})},R\right) \Big(1+\| u_x^{n+1/2}\|_{L^2([-R,R])} \Big),
    \end{equation}
    where $D_t^+$ is the forward time difference operator 
    \begin{equation*}
        D_t^+ u^n= \frac{u^{n+1}-u^n}{\Delta t}.
    \end{equation*}
\end{lemma}
\begin{proof}
    The scheme \eqref{CNscheme} can be rewritten as
    \begin{equation}\label{recn}
      \left\langle D_t^+ u^n , \phi v \right\rangle  = \left\langle \frac{(u^{n+1/2})^2}{2},(\phi v)_x\right\rangle - \left\langle \mathcal{D}^{\alpha/2} u^{n+1/2} , \mathcal{D}^{\alpha/2}(\phi v)_x\right\rangle, \quad v\in S_{\Delta x}.
    \end{equation}
   We choose $v=P(\psi)$ in \eqref{recn}, where $\psi\in C_c^\infty(-R,R)$ and $P$ is the projection operator described in the Lemma \ref{Lemma:projection}, to obtain
    \begin{equation}\label{temp:t_der_1}
      \left\langle D_t^+ u^n , \phi P(\psi) \right\rangle  = \left\langle \frac{(u^{n+1/2})^2}{2},(\phi P(\psi))_x\right\rangle - \left\langle \mathcal{D}^{\alpha/2} u^{n+1/2} , \mathcal{D}^{\alpha/2}(\phi P(\psi))_x\right\rangle.
    \end{equation}
   Using the estimates in \eqref{Proj}, we have
    \begin{equation*}
        \begin{split}
            \int_\mathbb{R} \left(u^{n+1/2}\right)^2 & (\phi P(\psi))_x\,dx\\
            & \leq \Big(\|P(\psi)\|_{L^\infty([-R,R])} + \|P(\psi)_x\|_{L^\infty([-R,R])}(2+2R) \Big)\int_{-R}^R \left(u^{n+1/2}\right)^2 \,dx\\
            & \leq \Big(\|P(\psi)\|_{H^1([-R,R])} + \|P(\psi)_x\|_{H^1([-R,R])}(2+2R)\Big) \left\|u^{n+1/2}\right\|^2_{L^2(\mathbb{R})} \\
            &\leq  C(\norm{u_0}_{L^2(\mathbb{R})}\|,R)\|\psi\|_{H^2([-R,R])},
        \end{split}
    \end{equation*}
where we have used the Sobolev inequality and the bound \eqref{Stab}.
We also deduce the following estimate 
    \begin{equation*}
        \begin{split}
            - \left\langle \mathcal{D}^{\alpha/2} u^{n+1/2} , \mathcal{D}^{\alpha/2} (\phi P(\psi))\right\rangle & 
            =  \left\langle u_x^{n+1/2} , \mathcal{D}^{\alpha} (\phi P(\psi))_x\right\rangle \\
            &\leq \| u_x^{n+1/2}\|_{L^2([-R,R])}\| \mathcal{D}^{\alpha}(\phi P(\psi))\|_{L^2([-R,R])}\\  
            & \leq C( \| u_0\|_{L^2(\mathbb{R})},R) \| u_x^{n+1/2}\|_{L^2([-R,R])} \|  P(\psi)\|_{H^{\alpha}([-R,R])}\\  
           % & \leq C( \| u_0\|_{L^2(\mathbb{R})},R)\| (\phi P(\psi))_x\|_{H^{\alpha/2}([-R,R])}\\  
            %& \leq C( \| u_0\|_{L^2(\mathbb{R})},R)\| P(\psi)\|_{H^{1+\alpha/2}([-R,R])}\\
             %& \leq C( \| u_0\|_{L^2(\mathbb{R})},R)\| \psi\|_{H^{1+\alpha/2}([-R,R])}\\
             & \leq C( \| u_0\|_{L^2(\mathbb{R})},R) \| u_x^{n+1/2}\|_{L^2([-R,R])}  \| \psi\|_{H^{2}([-R,R])},\\
        \end{split}
    \end{equation*}
  where we have used \eqref{siq1}, and estimates \eqref{fracP} and \eqref{Dalp}.
 Combining the above estimates together in \eqref{temp:t_der_1}, we end up with 
    \begin{align*}
    \left| \int_{-R}^{R} D_t^+ u^n \phi P(\psi)\,dx \right| & =  \left|\int_{-R}^{R} D_t^+ u^n \phi \psi\,dx \right|  \\
    & \leq C\left( \| u_0\|_{L^2(\mathbb{R})},R\right) \Big(1+\| u_x^{n+1/2}\|_{L^2([-R,R])} \Big) \|\psi\|_{H^2([-R,R])},
    \end{align*}
   and the estimate \eqref{TempDB} follows.
%    \newline
%    If $\psi\in C_c^\infty(\mathbb{R})$ then $P(\psi)\in S_{\Delta x}$. By the exact same arguments as above, but this time on $\mathbb{R}$ instead of $[-R,R]$, we get 
%    \begin{equation}\label{tempo}
%        \|D_t^+ u^n \phi\|_{H^{-2}(\mathbb{R})} \leq C\left( \| u_0\|_{L^2(\mathbb{R})},R\right),
%    \end{equation}
\end{proof}
\subsection{Convergence to a weak solution.}
We start by defining the weak solution of the Cauchy problem \eqref{fkdv}.
\begin{definition}\label{defn:weak_soln}
    Let $Q>0$ be any given number and $u_0\in L^2(\mathbb{R})$. Then $u\in L^2(0,T;H^{\alpha/2}(-Q,Q))$ is said to be a weak solution of \eqref{fkdv} in the region $(-Q,Q)\times[0,T)$ if 
    \begin{equation}\label{weaks}
        \int_0^T \int_{\mathbb{R}}\left(\varphi_t u + \varphi_x\frac{u^2}{2} - \mathcal{D}^{\alpha/2} \varphi_x \mathcal{D}^{\alpha/2}u\right)\,dx\,dt + \int_{\mathbb{R}}\varphi(x,0)u_0(x)\,dx = 0,
    \end{equation}
    for all $\varphi\in C_c^\infty((-Q,Q)\times[0,T)) $.
\end{definition}
Next we define the approximate solution $u_{\Delta x}\in S_{\Delta x}$ by the interpolation formula
\begin{equation}\label{appsol}
u_{\Delta x}(x,t) = 
    \begin{cases}
        u^{n-1/2}(x)+(t-t_{n-1/2})D^+_t u^{n-1/2}(x), \qquad &t\in [t_{n-1/2},t_{n+1/2}),n\geq 1 ,\\
        u^0(x) + 2t\frac{u^{1/2}(x)-u^0(x)}{\Delta t}, \qquad &t\in[0,t_{1/2}).
    \end{cases}
\end{equation}
Hereby, our aim is to demonstrate the convergence of $u_{\Delta x}(x,t)$ to a weak solution of \eqref{fkdv}.
More precisely, we have the following result.
 \begin{theorem}
     Let $u_0 \in {L^2(\mathbb{R})}$ and $\{u^n\}_{n\in\mathbb{N}}$ be a sequence of function defined by the scheme \eqref{CNscheme}. Further, assume that $\Delta t = \mathcal{O}(\Delta x^2)$. Then there exist a time $T>0$ and a contant $C$, depending only on $R$ and $\norm{u_0}_{L^2(\mathbb{R})}$ such that 
      \begin{align}
         \norm{u_{\Delta x}}_{L^\infty(0,T;L^2([-R,R]))}&\leq C(\norm{u_0}_{L^2(\mathbb{R})},R), \label{b1}\\
        \norm{u_{\Delta x}}_{L^2(0,T;H^{\alpha/2}([-R,R]))}&\leq C(\norm{u_0}_{L^2(\mathbb{R})},R), \label{b2}\\
         \norm{\partial_t (u_{\Delta x} \phi)}_{L^2(0,T;H^{-2}([-R,R]))} &\leq C(\norm{u_0}_{L^2(\mathbb{R})},R), \label{b3}
      \end{align}
     where $u_{\Delta x}$ is described by \eqref{appsol}. Moreover, there exist a sequence $\{\Delta x_j\}_{j=1}^\infty$ and 
     as $\Delta x_j \xrightarrow[j\xrightarrow{}\infty]{} 0$
     \begin{equation}\label{congg1}
         u_{\Delta x_j}\xrightarrow[]{} u \hspace{0.1cm} strongly \hspace{0.1cm} in \hspace{0.1cm} L^2(0,T;L^2([-R,R])),
     \end{equation}
     where $u \in L^2(0,T;L^2([-R,R]))$ is a weak solution of the Cauchy problem \eqref{fkdv} in the sense of Definition \ref{defn:weak_soln} for $Q=R-1$.
 \end{theorem}
% \begin{remark}
%     The Convergent subsequence $u^{\Delta x_j}$ in above theorem can be used as a constructive proof of existence of solutions to \eqref{fkdv}, as noted in section 1. On the other hand, owing to well-posedness for the initial data $u_0\in L^2(\mathbb{R})$ \cite{RajibBO} we can infact conclude that the whole sequence converges as $\Delta x\xrightarrow[]{}0$.
% \end{remark}
 \begin{proof}
    Let us consider $T=(N+1/2)\Delta t $ for some $N\in \mathbb{N}$. Then $u_{\Delta x}(x,t)$ can be represented by
     \begin{equation*}
         u_{\Delta x}(x,t) = (1-\alpha_n(t))u^{n-1/2}(x) + \alpha_n(t)u^{n+1/2}(x), \quad t\in [t_{n-1/2},t_{n+1/2}),
     \end{equation*}
     where $\alpha_n\in [0,1)$ is given by $\alpha_n = (t - t_{n-1/2})/\Delta t$.
     
   For $t\in [t_{n-1/2}, t_{n+1/2}), n= 1,2...,N,$ using \eqref{Stab}, we have
     \begin{equation*}
         \norm{u_{\Delta x}}_{L^2(\mathbb{R})} \leq \norm{u^{n-1/2}}_{L^2(\mathbb{R})} + \norm{u^{n+1/2}}_{L^2(\mathbb{R})}\leq C\left(\norm{u_0}_{L^2(\mathbb{R})}\right).
     \end{equation*}
Again, for $t\in[0,t_{1/2})$, we have
     \begin{equation*}
         \norm{u_{\Delta x}}_{L^2(\mathbb{R})} \leq |1-(2t)/\Delta t | \norm{u^0}_{L^2(\mathbb{R})} + |2t/\Delta t| \|u^
         {1/2}\|_{L^2(\mathbb{R})}\leq C\left(\norm{u_0}_{L^2(\mathbb{R})}\right).
     \end{equation*}
     This proves \eqref{b1}.
     Next we have
     \begin{align*}
         \int_0^T & \norm{\mathcal{D}^{\alpha/2}u_{\Delta x}}^2_{L^2([-R,R])}\,dt \\
           &  \leq 2\norm{\mathcal{D}^{\alpha/2}u^0}^2_{L^2([-R,R])}\int_0^{t_{1/2}} \left(1-\frac{2t}{\Delta t}\right)^2\,dt
               +2\norm{\mathcal{D}^{\alpha/2}u^0}^2_{L^2([-R,R])}\int_0^{t_{1/2}} \left(\frac{2t}{\Delta t}\right)^2\,dt\\
               &\quad +2\sum_{n=1}^N\norm{\mathcal{D}^{\alpha/2}u^{n-1/2}}^2_{L^2([-R,R])}\int_{t_{n-1/2}}^{t_{n+1/2}} \left(1-\alpha_n(t)\right)^2\,dt\\
               &\quad +2\sum_{n=1}^N\norm{\mathcal{D}^{\alpha/2}u^{n+1/2}}^2_{L^2([-R,R])}\int_{t_{n-1/2}}^{t_{n+1/2}} \left(\alpha_n(t)\right)^2\,dt\\
               & \leq  C \Delta t  \norm{u^0}^2_{H^1([-R,R])} + 2\Delta t \sum_{n=1}^N\norm{\mathcal{D}^{\alpha/2}u^{n+1/2}}^2_{L^2([-R,R])}\\
             &  = C \Delta t  \norm{u^0}^2_{L^2([-R,R])} + C\Delta t \|u^0_x\|^2_{L^2([-R,R])}+ 2\Delta t \sum_{n=1}^N\norm{\mathcal{D}^{\alpha/2}u^{n+1/2}}^2_{L^2([-R,R])},
     \end{align*}
     where we have taken into account the Sobolev inequality $\norm{u^0}_{H^{\alpha/2}([-R,R])}\leq C \norm{u^0}_{H^{1}([-R,R])}$.
     By the help of \eqref{Dalp} and the inverse inequality \eqref{inv_inequality} along with the assumption $\Delta t\leq C\Delta x^2$, provide the estimate \eqref{b2}.
   
   We observe that
     \begin{align*}
        \partial_t u_{\Delta x} &= 
%         \begin{cases}
%             D_t^+u^{n-1/2}, \quad &(x,t)\in \mathbb{R} \times [t_{n-1/2},t_{n+1/2}),\\
%             \frac{u^{1/2}-u^0}{\Delta t/2}, \quad &(x,t)\in \mathbb{R} \times [0,t_{1/2}),
%         \end{cases}
%        =
         \begin{cases}
            D_t^+u^{n-1/2} = \frac{D_t^+u^{n}+D_t^+u^{n-1}}{2}, \qquad &(x,t)\in \mathbb{R} \times [t_{n-1/2},t_{n+1/2}),\\
             \frac{u^{1/2}-u^0}{\Delta t/2} = D^t_+u^0, \qquad &(x,t)\in \mathbb{R} \times [0,t_{1/2}).
         \end{cases}
          \end{align*}
 Thus, using the bounds \eqref{TempDB}, \eqref{inv_inequality} and \eqref{Stab}, we get
 \begin{align*}
    \int_0^T \norm{\partial_t(u_{\Delta x} \phi)}^2_{H^{-2}(\mathbb{R})} & \leq 
    \int_0^{t_{1/2}} \norm{D^t_+u^0\phi}^2_{H^{-2}(\mathbb{R})}\,dt \\ & \qquad + \frac{1}{2}
     \sum_{n=1}^{N}\int_{t_{n-1/2}}^{t_{n+1/2}} \norm{D_t^+u^{n}\phi}^2_{H^{-2}(\mathbb{R})}+\norm{D_t^+u^{n-1}\phi}^2_{H^{-2}(\mathbb{R})}\,dt\nonumber\\
      &\leq 
    C\int_0^{t_{1/2}} \norm{u^{1/2}_x}^2_{L^2([-R,R])}\,dt \\ & \qquad + \frac{C}{2}
     \sum_{n=1}^{N}\int_{t_{n-1/2}}^{t_{n+1/2}} \norm{u^{n+1/2}_x}^2_{L^2([-R,R])}+\norm{u^{n-1/2}_x}^2_{L^2([-R,R])}\,dt\nonumber\\
     &\leq 
    C \frac{\Delta t}{2\Delta x^2}\norm{u^{1/2}}^2_{L^2([-R,R])} \\ & \qquad + C\frac{\Delta t}{2\Delta x^2}
     \sum_{n=1}^{N}\norm{u^{n+1/2}}^2_{L^2([-R,R])}+\norm{u^{n-1/2}}^2_{L^2([-R,R])}\nonumber\\
   &\leq C \left(\norm{u_0}_{L^2(\mathbb{R})},R\right).
 \end{align*}
 Considering the properties of $\phi$ and using the estimates \eqref{b1} and \eqref{b2},  we obtain
 \begin{align}
     \label{c1}\|\phi u_{\Delta x}\|_{L^\infty(0,T;L^2([-R,R]))}&\leq C\left(\norm{u_0}_{L^2(\mathbb{R})},R\right),\\
     \label{c2}\|\phi u_{\Delta x}\|_{L^2(0,T;H^{\alpha/2}([-R,R]))}&\leq C\left(\norm{u_0}_{L^2(\mathbb{R})},R\right).
 \end{align}
 Based on the bounds \eqref{c1}, \eqref{c2} and \eqref{b3}, we apply the Aubin-Simon compactness lemma (cf. \cite[Lemma 4.4]{HOLDEN1999})  to the set $\{\phi u_{\Delta x}\}$. Then there exists a sequence $\{\Delta x_j\}_{j\in\mathbb{N}}$ such that $\Delta x_j \xrightarrow[]{} 0$ as $j$ tends to infinity, and a function $\Tilde{u}$ such that
 \begin{equation}\label{congg}
         \phi u_{\Delta x_j}\longrightarrow \Tilde{u} \hspace{0.1cm} \text{strongly in }  L^2(0,T;L^2([-R,R])),
     \end{equation}
     and subsequently, \eqref{congg1} holds. The strong convergence ensures the limit can be passed through the nonlinear term.
     
     Our next aim is to show that $u$ is, in fact, a weak solution to \eqref{fkdv}. The standard $L^2$-projection of a function $\psi \in C^{k+1}(\R)$ onto the finite element space $S_{\Delta x}$ for some $k\in \mathbb{N}_0$, denoted by $\mathcal{P}$, satisfies
     \begin{equation*}
         \int_{\mathbb{R}} (\mathcal{P}\psi(x)-\psi(x))v(x)\,dx = 0 ,\quad \forall v\in S_{\Delta x}.
     \end{equation*}
    Furthermore, the projection operator $\mathcal{P}$ satisfies (see Ciarlet \cite{Ciarlet})
     \begin{equation}\label{projinq}
         \|\psi(x)-\mathcal{P}\psi(x)\|_{H^k(\mathbb{R})}\leq C\Delta x\|\psi\|_{H^{k+1}(\mathbb{R})},
     \end{equation}
     where the constant $C$ is independent of $\Delta x$. 
    
     For $v\in S_{\Delta x}$ and $n\geq 1$, taking into account \eqref{recn}, we have
     \begin{align*}
      \left\langle D_t^+ u^n , \phi v \right\rangle  -\left\langle \frac{(u^{n+1/2})^2}{2},(\phi v)_x\right\rangle + \left\langle \mathcal{D}^{\alpha/2} u^{n+1/2} , \mathcal{D}^{\alpha/2}(\phi v)_x\right\rangle &= 0,\\
      \left\langle D_t^+ u^{n-1} , \phi v \right\rangle  -\left\langle \frac{(u^{n-1/2})^2}{2},(\phi v)_x\right\rangle + \left\langle \mathcal{D}^{\alpha/2} u^{n-1/2} , \mathcal{D}^{\alpha/2}(\phi v)_x\right\rangle &= 0.
    \end{align*}
    By taking average of the above two relations gives
    \begin{align}
        \mathcal{G}_n(\phi v):=  \left\langle D_t^+ u^{n-1/2} , \phi v \right\rangle  &-\left\langle \frac{(u^{n+1/2})^2+(u^{n-1/2})^2}{2},(\phi v)_x\right\rangle \nonumber\\ 
        & \qquad+ \left\langle \mathcal{D}^{\alpha/2} \frac{u^{n+1/2}+ u^{n-1/2}}{2}, \mathcal{D}^{\alpha/2}(\phi v)_x\right\rangle = 0.
    \end{align}
   Our aim is to demonstrate that
    \begin{equation}\label{Orede1}
        \int_0^T\int_\mathbb{R} \left(\partial_t u_{\Delta x}\phi v - \frac{(u_{\Delta x})^2}{2}(\phi v)_x -(\mathcal{D}^{\alpha/2} u_{\Delta x})\mathcal{D}^{\alpha/2}(\phi v)_x\right)\,dx\,dt = \mathcal{O}(\Delta x),
    \end{equation}
    for any test function $v\in C_c^\infty((-R+1,R-1)\times [0,T))$, where $\phi$ is specified in Section \ref{sec2}. We proceed as follows.
    \begin{align*}
        \int_0^T&\int_\mathbb{R} \left(\partial_t u_{\Delta x}\phi v - \frac{(u_{\Delta x})^2}{2}(\phi v)_x 
        -(\mathcal{D}^{\alpha/2} u_{\Delta x})\mathcal{D}^{\alpha/2}(\phi v)_x\right)\,dx\,dt\\
         = &\int_0^{t_{1/2}}\int_\mathbb{R} \left(\partial_tu_{\Delta x}\phi v - \frac{(u_{\Delta x})^2}{2}(\phi v)_x -(u_{\Delta x})\mathcal{D}^{\alpha}(\phi v)_x\right)\,dx\,dt\\
         &+ \sum_{n=1}^N \int_{t_{n-1/2}}^{t_{n+1/2}}\int_\mathbb{R} \left(\partial_tu_{\Delta x}\phi v - \frac{(u_{\Delta x})^2}{2}(\phi v)_x -(u_{\Delta x})\mathcal{D}^{\alpha}(\phi v)_x\right)\,dx\,dt  \\
         =:& \mathcal{I}_{\Delta x} + \mathcal{E}_{\Delta x}.
    \end{align*}
   We choose the test function as  $v^{\Delta x} := \mathcal{P}v$ and observe that  $ \mathcal{I}_{\Delta x} $ and $\mathcal{E}_{\Delta x}$ can be rewritten as
   \begin{align*}
       \mathcal{I}_{\Delta x} =& \int_0^{t_{1/2}}\underbrace{\int_\mathbb{R} \left((D^+_tu^0)\phi v^{\Delta x} - \frac{(u^{1/2})^2}{2}(\phi v^{\Delta x})_x -(\mathcal{D}^{\alpha/2} u^{1/2})\mathcal{D}^{\alpha/2}(\phi v^{\Delta x})_x\right)\,dx}_{=0}\,dt\\
       &+ \underbrace{\int_0^{t_{1/2}}\int_\mathbb{R}(D^+_tu^0)\phi(v- v^{\Delta x}) \,dx\,dt}_{\mathcal{I}_1^{\Delta x}}
       - \underbrace{ \int_0^{t_{1/2}}\int_\mathbb{R} \frac{(u^{1/2})^2}{2}(\phi (v-v^{\Delta x}))_x\,dx\,dt}_{\mathcal{I}_2^{\Delta x}}\\
       &- \underbrace{ \int_0^{t_{1/2}}\int_\mathbb{R} \left[\frac{(u_{\Delta x})^2}{2}-\frac{(u^{1/2})^2}{2}\right](\phi v)_x\,dx\,dt}_{\mathcal{I}_3^{\Delta x}}
        - \underbrace{ \int_0^{t_{1/2}}\int_\mathbb{R} \left(u^{1/2}\right)\mathcal{D}^{\alpha}(\phi (v-v^{\Delta x}))_x\,dx\,dt}_{\mathcal{I}_4^{\Delta x}}\\
        &- \underbrace{\int_0^{t_{1/2}}\int_\mathbb{R} \left(u_{\Delta x}-u^{1/2}\right)\mathcal{D}^{\alpha}(\phi v)_x\,dx\,dt}_{\mathcal{I}_5^{\Delta x}},\\
   \end{align*}
   and
   \begin{align*}
       \mathcal{E}_{\Delta x} = & \sum_{n=1}^N \int_{t_{n-1/2}}^{t_{n+1/2}}\underbrace{\mathcal{G}_n(\phi v^{\Delta x})}_{=0}\,dt + \underbrace{\sum_{n=1}^N \int_{t_{n-1/2}}^{t_{n+1/2}}\int_\mathbb{R} \left(D^+_tu^{n-1/2}\right)\phi(v- v^{\Delta x})\,dx\,dt}_{\mathcal{E}_1^{\Delta x}} \\
       &-\underbrace{\sum_{n=1}^N \int_{t_{n-1/2}}^{t_{n+1/2}}\int_\mathbb{R} \frac{1}{2}\frac{(u^{n+1/2})^2+(u^{n-1/2})^2}{2}(\phi(v- v^{\Delta x}))_x\,dx\,dt}_{\mathcal{E}_2^{\Delta x}} \\
       &- \underbrace{\sum_{n=1}^N \int_{t_{n-1/2}}^{t_{n+1/2}}\int_\mathbb{R} \frac{1}{2}\left[(u_{\Delta x})^2 -\frac{(u^{n+1/2})^2+(u^{n-1/2})^2}{2}\right](\phi v)_x\,dx\,dt}_{\mathcal{E}_3^{\Delta x}} \\
       &- \underbrace{\sum_{n=1}^N \int_{t_{n-1/2}}^{t_{n+1/2}}\int_\mathbb{R} \left(\frac{u^{n+1/2}+u^{n-1/2}}{2}\right)\mathcal{D}^{\alpha}(\phi(v- v^{\Delta x}))_x\,dx\,dt}_{\mathcal{E}_4^{\Delta x}} \\
       &- \underbrace{\sum_{n=1}^N \int_{t_{n-1/2}}^{t_{n+1/2}}\int_\mathbb{R} \left(u_{\Delta x}-\frac{u^{n+1/2}+u^{n-1/2}}{2}\right)\mathcal{D}^{\alpha}(\phi v)_x\,dx\,dt}_{\mathcal{E}_5^{\Delta x}}.
   \end{align*}
   We estimate the terms $\mathcal{I}_i^{\Delta x}$ and $\mathcal{E}_i^{\Delta x},~i=1,2,3,4,5$. From \eqref{siq1}, \eqref{projinq} and \eqref{b3}, we obtain
   \begin{align*}
       \mathcal{I}_1^{\Delta x} + \mathcal{E}_1^{\Delta x} &=  \int_0^T\int_{-R}^R \pl_t u_{\Delta x}\phi (v - v^{\Delta x})\,dx\,dt\\
       &\leq  \int_0^T \|\partial_t u_{\Delta x}\phi \|_{H^{-2}([-R,R])}\|v - v^{\Delta x}\|_{H^{2}([-R+1,R-1])}\,dt\\
       &\leq C\Delta x  \left(\norm{u_0}_{L^2(\mathbb{R})},R\right)\norm{v}_{L^2(0,T;H^{3}([-R+1,R-1]))} \xrightarrow[]{\Delta x\xrightarrow[]{}0}0.
   \end{align*}
   From \eqref{Stab} and \eqref{projinq}, we get
   \begin{align*}
       \mathcal{I}_2^{\Delta x} + \mathcal{E}_2^{\Delta x} \leq&  \frac{1}{2}\int_0^{t_{1/2}}\int_{-R+1}^{R-1} \left|u^{1/2}\right|^2|(\phi (v-v^{\Delta x}))_x|\,dx\,dt\\
       &+ \frac{1}{4}\sum_{n=1}^N \int_{t_{n-1/2}}^{t_{n+1/2}}\int_{-R+1}^{R-1} \left(\left|u^{n+1/2}\right|^2+\left|u^{n-1/2}\right|^2\right)|(\phi(v- v^{\Delta x}))_x|\,dx\,dt\\
       \leq & C(R)\int_0^{t_{1/2}} \|v-v^{\Delta x}\|_{H^2([-R+1,R-1])} \|u^{1/2}\|^2_{L^2([-R,R])}\,dt\\
       &+ C(R)\sum_{n=1}^N \int_{t_{n-1/2}}^{t_{n+1/2}}\|v-v^{\Delta x}\|_{H^2([-R+1,R-1])} \norm{u^{n+1/2}}^2_{L^2([-R,R])}\,dt\\
       &+C(R)\sum_{n=1}^N \int_{t_{n-1/2}}^{t_{n+1/2}}\|v-v^{\Delta x}\|_{H^2([-R+1,R-1])} \norm{u^{n-1/2}}^2_{L^2([-R,R])}\,dt\\
       \leq & C\Delta x  \left(\norm{u_0}_{L^2(\mathbb{R})},R\right)\norm{v}_{L^\infty(0,T;H^3([-R+1,R-1]))} \xrightarrow[]{\Delta x\xrightarrow[]{}0}0.
   \end{align*}
   The next terms rewritten as 
   \begin{align*}
       \mathcal{I}_3^{\Delta x} + \mathcal{E}_3^{\Delta x} = & \frac{\Delta t}{4} \int_0^{t_{1/2}}\int_{-R+1}^{R-1} \left(u^{1/2}+u^0\right)(D^+_tu^0)(\phi v)_x\,dx\,dt \\
       &- \int_0^{t_{1/2}}\int_{-R+1}^{R-1} \left(u^0t(D^+_tu^0)(\phi v)_x +\frac{1}{2}t^2(D^+_tu^0)^2(\phi v)_x \right)\,dx\,dt \\
       &+\frac{\Delta t}{4} \sum_{n=1}^N \int_{t_{n-1/2}}^{t_{n+1/2}}\int_{-R+1}^{R-1} \left(u^{n+1/2}+u^{n-1/2}\right)\left(D_t^+u^{n-1/2}\right)(\phi v)_x\,dx\,dt\\
       &-\sum_{n=1}^N \int_{t_{n-1/2}}^{t_{n+1/2}}\int_{-R+1}^{R-1} u^{n-1/2}(t-t_{n-1/2})\left(D_t^+u^{n-1/2}\right)(\phi v)_x\,dx\,dt\\
       &-\frac{1}{2}\sum_{n=1}^N \int_{t_{n-1/2}}^{t_{n+1/2}}\int_{-R+1}^{R-1} (t-t_{n-1/2})^2\left(D_t^+u^{n-1/2}\right)^2(\phi v)_x\,dx\,dt
   \end{align*}
   Using \eqref{siq1}, \eqref{TempDB},\eqref{Stab} and Lemma 4.1 in \cite{CNGALTUNG} we have the estimates
   \begin{align*}
       \int_0^{t_{1/2}}&\int_{-R+1}^{R-1} u^0(D^+_tu^0)(\phi v)_x \,dx\,dt\\
       &+ \sum_{n=1}^N \int_{t_{n-1/2}}^{t_{n+1/2}}\int_{-R+1}^{R-1} u^{n-1/2}\left(D_t^+u^{n-1/2}\right)(\phi v)_x\,dx\,dt\\
       \leq & \int_0^{t_{1/2}} \norm{u^0}_{L^\infty([-R+1,R-1])}\|D^+_tu^0\phi\|_{H^{-2}([-R+1,R-1])}\norm{(\phi v)_x}_{H^{2}([-R+1,R-1])}\,dt\\
       &+ \sum_{n=1}^N \int_{t_{n-1/2}}^{t_{n+1/2}} \norm{u^{n-1/2}}_{L^\infty([-R,R])}\|D^+_tu^{n-1/2}\phi\|_{H^{-2}([-R+1,R-1])}\norm{(\phi v)_x}_{H^{2}([-R+1,R-1])}\,dt\\
       \leq & C \left(\norm{u_0}_{L^2(\mathbb{R})},R\right)\int_0^{t_{1/2}}\frac{C}{\sqrt{\Delta x}} \norm{u^0}_{L^2(\mathbb{R})}\norm{\phi v}_{H^{3}([-R+1,R-1])}\,dt\\
       &+C \left(\norm{u_0}_{L^2(\mathbb{R})},R\right)\sum_{n=1}^N \int_{t_{n-1/2}}^{t_{n+1/2}}\frac{C}{\sqrt{\Delta x}} \norm{u^{n-1/2}}_{L^2(\mathbb{R})}\norm{\phi v}_{H^{3}([-R+1,R-1])}\,dt\\
       \leq & \frac{1}{\sqrt{\Delta x}}C \left(\norm{u_0}_{L^2(\mathbb{R})},R\right)\norm{v}_{L^\infty(0,T;H^3([-R+1,R-1]))}.
   \end{align*}
   Similarly we can estimate
   \begin{align*}
       \int_0^{t_{1/2}}&\int_{-R+1}^{R-1} u^{1/2}(D^+_tu^0)(\phi v)_x\,dx\,dt\\
       &+ \sum_{n=1}^N \int_{t_{n-1/2}}^{t_{n+1/2}}\int_{-R+1}^{R-1} u^{n+1/2}\left(D_t^+u^{n-1/2}\right)(\phi v)_x\,dx\,dt\\
       \leq & \frac{1}{\sqrt{\Delta x}}C \left(\norm{u_0}_{L^2(\mathbb{R})},R\right)\norm{v}_{L^\infty(0,T;H^3([-R+1,R-1]))}.
   \end{align*}
   Furthermore we have 
   \begin{align*}
       \Big | \int_0^{t_{1/2}}& \int_{-R+1}^{R-1} u^0t(D^+_t u^0)(\phi v)_x\,dx\,dt \Big | \\&
       \qquad+ \left| \int_{t_{n-1/2}}^{t_{n+1/2}}\int_{-R+1}^{R-1} u^{n-1/2}(t-t_{n-1/2})\left(D_t^+u^{n-1/2}\right)(\phi v)_x\,dx\,dt\right|\\
       \leq& \frac{\Delta t}{2}\int_0^{t_{1/2}}\norm{u^0}_{L^\infty([-R+1,R-1])}\int_{-R+1}^{R-1} |D^+_tu^0| |(\phi v)_x|\,dx\,dt\\
       & \qquad + \Delta t\int_{t_{n-1/2}}^{t_{n+1/2}}\norm{u^{n-1/2}}_{L^\infty([-R+1,R-1])}\int_{-R+1}^{R-1} \left|D_t^+u^{n-1/2}\right|
      |(\phi v)_x|\,dx\,dt
   \end{align*}
   and
   \begin{align*}
       \Big |\int_0^{t_{1/2}}&\int_{-R+1}^{R-1} t^2(D^+_tu^0)^2(\phi v)_x\,dx\,dt \Big| \\
       +& \left|\int_{t_{n-1/2}}^{t_{n+1/2}}\int_{-R+1}^{R-1} (t-t_{n-1/2})^2\left(D_t^+u^{n-1/2}\right)^2(\phi v)_x\,dx\,dt\right|\\
       \leq& \frac{\Delta t}{2}\int_0^{t_{1/2}}\|u^{1/2}-u^0\|_{L^\infty([-R+1,R-1])}\int_{-R+1}^{R-1} |D^+_tu^0||(\phi v)_x|\,dx\,dt\\
       +& \Delta t\int_{t_{n-1/2}}^{t_{n+1/2}}\|u^{n+1/2}+u^{n-1/2}\|_{L^\infty([-R+1,R-1])}\int_{-R+1}^{R-1} \left|D_t^+u^{n-1/2}\right||(\phi v)_x|,dx\,dt.
   \end{align*}
   Now we can estimate these terms like the preceding terms, and as $\Delta t \leq C\Delta x^2$ we obtain
   \begin{equation*}
       \mathcal{I}_3^{\Delta x} + \mathcal{E}_3^{\Delta x} \xrightarrow[]{\Delta x\xrightarrow[]{}0}0.
   \end{equation*}
   Using \eqref{Dalp}, \eqref{Stab}, \eqref{b2} and \eqref{projinq} we obtain
   \begin{align*}
       \mathcal{I}_4^{\Delta x} + \mathcal{E}_4^{\Delta x} \leq  & \int_0^{t_{1/2}} \norm{u^{1/2}}_{L^2(\mathbb{R})}\norm{\phi(v- v^{\Delta x})}_{H^{1+\alpha}([-R+1,R-1])}\,dt\\
       +& \frac{1}{2}\sum_{n=1}^N \int_{t_{n-1/2}}^{t_{n+1/2}}\norm{u^{n+1/2}}_{L^2(\mathbb{R})}\|\phi(v- v^{\Delta x})\|_{H^{1+\alpha}([-R+1,R-1])}\,dt\\
       +& \frac{1}{2}\sum_{n=1}^N \int_{t_{n-1/2}}^{t_{n+1/2}}\norm{u^{n-1/2}}_{L^2(\mathbb{R})}\norm{\phi(v- v^{\Delta x})}_{H^{1+\alpha}([-R+1,R-1])}\,dt\\
       \leq & C \left(\norm{u_0}_{L^2(\mathbb{R})},R\right)\int_0^T \|v - v^{\Delta x}\|_{H^{3}([-R+1,R-1])}\,dt\\
       =& \Delta x C \left(\norm{u_0}_{L^2(\mathbb{R})},R\right)\norm{v}_{L^\infty(0,T;H^{4}([-R+1,R-1])} \xrightarrow[]{\Delta x\xrightarrow[]{}0}0.
   \end{align*}
   Finally
   \begin{align*}
       \mathcal{I}_5^{\Delta x} + \mathcal{E}_5^{\Delta x} = & \int_0^{t_{1/2}}\int_\mathbb{R}\Delta t\left(-\frac{1}{2}+\frac{t}{\Delta t}\right)(D^+_tu^0)(\mathcal{D}^{\alpha}(\phi v)_x)\,dx\,dt\\
       +& \sum_{n=1}^N \int_{t_{n-1/2}}^{t_{n+1/2}}\int_\mathbb{R}\Delta t\left(-\frac{1}{2}+\frac{t-t_{n-1/2}}{\Delta t}\right)(D^+_tu^{n-1/2})(\mathcal{D}^{\alpha}(\phi v)_x)\,dx\,dt\\
       \leq&\frac{\Delta t}{2} \int_0^{T}\int_\mathbb{R}|\partial_tu_{\Delta x}\phi||\mathcal{D}^\alpha(\phi v)_x|\,dx\,dt\\
       \leq&\frac{\Delta t}{2} \int_0^{T}\norm{\partial_tu_{\Delta x}\phi}_{H^{-2}(\R)}\|\mathcal{D}^\alpha(\phi v)_x\|_{H^{2}(\R)}\,dt\\
        \leq& \Delta t C \left(\norm{u_0}_{L^2(\mathbb{R})},R\right)\norm{\phi v}_{L^\infty(0,T;H^{3+\alpha}([-R+1,R-1]))} \xrightarrow[]{\Delta t\xrightarrow[]{}0}0.
   \end{align*}
 Combining the above estimates together to conclude that \eqref{Orede1} holds. Also, observe that by passing $\Delta x \xrightarrow{} 0$ we obtain
   \begin{equation}\label{Orede0}
         \int_0^T\int_\mathbb{R} \left(\partial_tu_{\Delta x}\phi v - \frac{(u_{\Delta x})^2}{2}(\phi v)_x -(\mathcal{D}^{\alpha/2} u_{\Delta x})\mathcal{D}^{\alpha/2}(\phi v)_x\right)\,dx\,dt =  0,
    \end{equation}
     for any test function $v\in C_c^\infty([-R+1,R-1]\times [0,T))$. Now choose $v=\varphi/\phi$ in \eqref{Orede0} with $\varphi\in C_c^\infty([-R+1,R-1]\times [0,T))$ and integrate by parts to conclude that \eqref{weaks} holds, that is
      \begin{equation*}
        \int_0^T \int_{\mathbb{R}}\left(\varphi_t u + \varphi_x\frac{u^2}{2} -(\mathcal{D}^{\alpha/2} \varphi_x)\mathcal{D}^{\alpha/2}u\right)\,dx\,dt + \int_{\mathbb{R}}\varphi(x,0)u_0(x)\,dx = 0.
    \end{equation*}
    This concludes the proof.
 \end{proof}
 \section{Convergence rate of the scheme}\label{sec5}
 Under the assumption that the initially data is sufficiently smooth, we determine the convergence rates for the devised scheme \eqref{CNscheme} of the associated initial value problem \eqref{fkdv}. More precisely, we prove the following result:
 %%%%%%%%%%%%%%%%%%%%%%%%%%
 \begin{theorem}\label{Thm5.1}
Assume that $u_0\in H^{\max{\{r+2,5\}}}$. Let $u$ be the solution of the IVP \eqref{fkdv} associated to the initial data $u_0$ and $u^n$ be the approximated solution obtained from the devised scheme \eqref{CNscheme} at time $t=t^n$. Then there holds
     \begin{equation}
         \norm{u^n - u(t_n)}_{L^2(\R)} = \mathcal{O} (\Delta x^{r-1} + \Delta t^2) , \qquad n = 0,1,\dots,N,
     \end{equation}
 where $r$ is the degree of the polynomial in the finite element space $S_{\Delta x}$.
 \end{theorem}
%%%%%%%%%%%%%%%%%%%%%%%
 In order to demonstrate the convergence rate in Theorem \ref{Thm5.1}, we require the following estimate \cite[p.98]{NumAppPDE}.
 \begin{lemma}
     Let $P$ be the projection defined by \eqref{projprop} on $S_{\Delta x}\cap C_c(-R,R)$. Assume that $u\in H^{l+1}(\R)$. If $s:=\min\{l,r+1\}$, then
     \begin{equation}\label{POlyap}
     |Pu-u|_{H^m(\mathbb{R}^n)} \leq C \Delta x^{s+1-m}|u|_{H^{s+1}(\mathbb{R}^n)}, \qquad m=0,1,2,\dots,l,
     \end{equation}
 where the constant $C$ is independent of $\Delta x$.
 \end{lemma}
 \begin{proof}[Proof of Theorem \ref{Thm5.1}]
 To begin with, we decompose the error as follows:
 \begin{equation*}
     u^n-u(t_n) = (u^n - Pu(t_n)) + (Pu(t_n) - u(t_n)):= \chi^n + \rho^n.
 \end{equation*}
Furthermore, for simplicity, let us denote $w^n := Pu(t_n)$. 
Testing the equation \eqref{fkdv} with $\phi v$, where $v\in H^{1+\alpha}(\R)$, and performing the integration by parts, we get 
 \begin{equation}\label{weaktime}
     \langle \pl_t u(t) , \phi v \rangle - \frac{1}{2}\langle u(t)^2, (\phi v)_x\rangle + \langle \mathcal{D}^{\alpha} u(t), (\phi v)_x\rangle = 0, \qquad t\in(0,T].
 \end{equation}
Using \eqref{CNscheme}, \eqref{weaktime} and \eqref{projprop}, we have  
\begin{align*}
    \left\langle\frac{\chi^{n+1}-\chi^n}{\Delta t},\phi v\right\rangle =& \left\langle\frac{u^{n+1}-u^n}{\Delta t}, \phi v\right\rangle - \left\langle\frac{\omega^{n+1}-\omega^n}{\Delta t}, \phi v\right\rangle \nonumber\\
    =&\left\langle\frac{u^{n+1}-u^n}{\Delta t}, \phi v\right\rangle - \langle \pl_t u(t_{n+1/2}),\phi v\rangle 
    +\left\langle \underbrace{\pl_t u(t_{n+1/2})-\frac{u^{n+1}-u^n}{\Delta t}}_{\xi^{n+1/2}}, \phi v\right\rangle\nonumber\\
    =& \frac{1}{2}\left\langle(u^{n+1/2})^2-u(t_{n+1/2})^2,(\phi v)_x\right\rangle \nonumber\\
    & \qquad -\left\langle \mathcal{D}^{\alpha}(u^{n+1/2} - u(t_{n+1/2})),(\phi v)_x\right\rangle + \left\langle \xi^{n+1/2},\phi v\right\rangle,
\end{align*}
for all $v\in S_{\Delta x}$. In particular, we choose $v=\chi^{n+1/2}$ to obtain
\begin{align}\label{temp:rate_0}
    \frac{1}{2}\norm{\chi^{n+1}}^2_{2,\phi} =& \frac{1}{2}\norm{\chi^{n}}^2_{2,\phi} + \Delta t\Big[\frac{1}{2}\left\langle(u^{n+1/2})^2-u(t_{n+1/2})^2,(\phi \chi^{n+1/2})_x\right\rangle \nonumber\\
    & \qquad -\left\langle \mathcal{D}^{\alpha}(u^{n+1/2} - u(t_{n+1/2})),(\phi \chi^{n+1/2})_x\right\rangle + \left\langle \xi^{n+1/2},\phi \chi^{n+1/2} \right\rangle \Big].
\end{align}
Moreover, we decompose the error at the mid-point $t_{n+1/2}$ as 
\begin{align}\label{temp:rate_1}
    u^{n+1/2} - u(t_{n+1/2}) & = \chi^{n+1/2}+\rho^{n+1/2} + \underbrace{\frac{u(t_{n+1})+u(t_{n})}{2} - u(t_{n+1/2})}_{\sigma^{n+1/2}}\\
    & = \chi^{n+1/2} + w^{n+1/2} - u(t_{n+1/2}),\nonumber
\end{align}
and consequently, we have
\begin{equation}\label{temp:rate_2}
    \begin{split}
        (u^{n+1/2})^2-u(t_{n+1/2})^2 =& (\chi^{n+1/2})^2 + 2\chi^{n+1/2}\omega^{n+1/2} + (\omega^{n+1/2})^2- u(t_{n+1/2})^2\\
        =&  (\chi^{n+1/2})^2 + 2\chi^{n+1/2} w^{n+1/2} + (\omega^{n+1/2}+u(t_{n+1/2}))(\rho^{n+1/2} + \sigma^{n+1/2}). 
    \end{split}
\end{equation}
Since $\chi^{n+1/2}\in S_{\Delta x}$ and using the integration by parts repeatedly, we get the following identities:
\begin{equation}\label{temp:rate_3}
    \begin{split}
        \left\langle(\chi^{n+1/2})^2,(\phi \chi^{n+1/2})_x\right\rangle & = \frac{2}{3}\left\langle(\chi^{n+1/2})^3,\phi_x\right\rangle,\\
        2\left\langle \chi^{n+1/2}\omega^{n+1/2},(\phi \chi^{n+1/2})_x\right\rangle &= - \left\langle(\chi^{n+1/2})^2,\phi\omega_x^{n+1/2}\right\rangle+
        \left\langle(\chi^{n+1/2})^2,\phi_x\omega^{n+1/2}\right\rangle.
    \end{split}
\end{equation}
Using the identities \eqref{temp:rate_3} and taking into account of \eqref{temp:rate_1} and \eqref{temp:rate_2}, the equation \eqref{temp:rate_0} reduces to
\begin{align}\label{temp:rate_4}
        \frac{1}{2}\norm{\chi^{n+1}}^2_{2,\phi}& = \frac{1}{2}\norm{\chi^{n}}^2_{2,\phi} + \Delta t\Big[\frac{1}{3}\left\langle(\chi^{n+1/2})^3,\phi_x\right\rangle- 
    \frac{1}{2}\left\langle(\chi^{n+1/2})^2,\phi\omega_x^{n+1/2}\right\rangle \nonumber\\
&+\left\langle(\chi^{n+1/2})^2,\phi_x\omega^{n+1/2}\right\rangle+
        \frac{1}{2}\left\langle (\omega^{n+1/2}+u(t_{n+1/2}))(\rho^{n+1/2} + \sigma^{n+1/2}),(\phi \chi^{n+1/2})_x\right\rangle\nonumber\\
        & -\left\langle \mathcal{D}^\alpha(\chi^{n+1/2}+\rho^{n+1/2}+\sigma^{n+1/2}),(\phi \chi^{n+1/2})_x\right\rangle + \left\langle \xi^{n+1/2},\phi \chi^{n+1/2} \right\rangle \Big].
\end{align}
Let $f\in H^s(\mathbb{R})$ and $0<a<1$ be such that $s\geq 2a$. Then by the Fourier transform (for more details refer to \cite{Grubb}), we have
\begin{equation*}
    \norm{\mathcal{D}^{2a}f}^2_{L^2(\mathbb{R})} \leq C\norm{\mathcal{D}^{2a}f}^2_{H^{s-2a}(\mathbb{R})} \leq C\norm{f}^2_{H^{s}(\mathbb{R})}.
\end{equation*}
Hence by the estimates \eqref{B_1 esti}, \eqref{B_2_esti} and using the Cauchy-Schwarz inequality, the equation \eqref{temp:rate_4} turns into
%\begin{align*}
%\mathcal{B}_1:=& \int_\mathbb{R}w^3\chi_x\,dx \leq \frac{1}{2} %\|\sqrt{\chi_x}\mathcal{D}^{\alpha/2}w\|^2_{L^2(\mathbb{R})}+\frac{1}
%{2}C_S\norm{w}^2_{L^2(\mathbb{R})}+\frac{1}{2}
%(1+C^2\norm{w}^2_{L^2(\mathbb{R})})\norm{w}^2_{L^2(\mathbb{R})}\\
%\mathcal{B}_2:=& \int_\mathbb{R} \mathcal{D}^{\alpha/2}w \mathcal{D}^{\alpha/2}
%(\chi w)_x\,dx \geq \|\sqrt{\chi_x}\mathcal{D}^{\alpha/2}w\|^2
%\end{align*}
\begin{equation*}
    \begin{split}
        \frac{1}{2}\norm{\chi^{n+1}}^2_{2,\phi} + \Delta t&\norm{\sqrt{\phi_x} \mathcal{D}^{\alpha/2}\chi^{n+1/2}}^2_{L^2(\mathbb{R})} \leq\frac{1}{2}\norm{\chi^{n}}^2_{2,\phi} + \frac{\Delta t}{6}\norm{\sqrt{\phi_x} \mathcal{D}^{\alpha/2}\chi^{n+1/2}}^2_{L^2(\mathbb{R})}\\&+\frac{\Delta t}{6}\Bigg[C_S\norm{\chi^{n+1/2}}^2_{L^2(\mathbb{R})}+C\left(1+\norm{u^{n+1/2}-\omega^{n+1/2}}^2_{L^2(\mathbb{R})}\right)\norm{\chi^{n+1/2}}^2_{L^2(\mathbb{R})}\\&+\norm{\chi^{n+1/2}}^2_{2,\phi}\norm{\omega_x^{n+1/2}}_{L^\infty}+\norm{\chi^{n+1/2}}^2_{2,\phi}\norm{\omega^{n+1/2}}_{L^\infty(\R)}\\&
        +\|\omega^{n+1/2}+u(t_{n+1/2})\|_{L^\infty}\left(\norm{\rho_x^{n+1/2}}_{L^2(\R)}+ \norm{\sigma_x^{n+1/2}}_{L^2(\R)}\right)\norm{\chi^{n+1/2}}_{2,\phi}\\& +\|\omega_x^{n+1/2}+u_x(t_{n+1/2})\|_{L^\infty}\left(\norm{\rho^{n+1/2}}_{L^2(\R)}+\norm{\sigma^{n+1/2}}_{L^2(\R)}\right)\norm{\chi^{n+1/2}}_{2,\phi}\\&
        +\norm{\mathcal{D}^{\alpha}\rho_x^{n+1/2}}_{L^2(\mathbb{R})}\norm{\chi^{n+1/2}}_{L^2(\R)} + \norm{\mathcal{D}^{\alpha}\sigma_x^{n+1/2}}_{L^2(\mathbb{R})} \norm{\chi^{n+1/2}}_{2,\phi} \\&+ C_R\norm{\xi^{n+1/2}}_{L^2(\mathbb{R})}\norm{\chi^{n+1/2}}_{2,
        \phi} \Bigg],
    \end{split}
\end{equation*} 
 where $\norm{\sqrt{\phi}}_{L^{\infty}(\R)}\leq C_R$. Furthermore, using the Sobolev inequality $\norm{\omega}_{L^{\infty}(\mathbb{R})} \leq \norm{\omega}_{H^1(\mathbb{R})}$ and triangle inequality we deduce
 \begin{equation*}
    \begin{split}
        \frac{1}{2}\norm{\chi^{n+1}}&^2_{2,\phi} + \frac{5\Delta t}{6}\norm{\sqrt{\phi_x} \mathcal{D}^{\alpha/2}\chi^{n+1/2}}^2_{L^2(\mathbb{R})} \leq\frac{1}{2}\norm{\chi^{n}}^2_{2,\phi} \\
        +&C\frac{\Delta t}{6}\Bigg[ \norm{\chi^{n+1/2}}^2_{L^2(\mathbb{R})}+\left(1+\norm{u^{n+1/2}}^2_{L^2(\mathbb{R})}+\norm{\omega^{n+1/2}}^2_{L^2(\mathbb{R})}\right)\norm{\chi^{n+1/2}}^2_{L^2(\mathbb{R})}\\+&    \norm{\chi^{n+1/2}}^2_{2,\phi}\norm{\omega^{n+1/2}}_{H^2}+\norm{\chi^{n+1/2}}^2_{2,\phi}\norm{\omega^{n+1/2}}_{H^1(\R)}\\+& \left(\norm{\omega^{n+1/2}}_{H^1(\R)}+\norm{u(t_{n+1/2})}_{H^1(\R)}\right)\left(\norm{\rho_x^{n+1/2}}_{L^2(\R)}+ \norm{\sigma_x^{n+1/2}}_{L^2(\R)}\right)\norm{\chi^{n+1/2}}_{2,\phi}\\ +&\left(\norm{\omega_x^{n+1/2}}_{H^1(\R)}+\norm{u_x(t_{n+1/2})}_{H^1(\R)}\right)\left(\norm{\rho^{n+1/2}}_{L^2(\R)}+ \norm{\sigma^{n+1/2}}_{L^2(\R)}\right)\norm{\chi^{n+1/2}}_{2,\phi}\\+&\norm{\rho^{n+1/2}}_{H^{3}(\mathbb{R})}\norm{\chi^{n+1/2}}_{2,\phi}+\norm{\sigma^{n+1/2}}_{H^{3}(\mathbb{R})} \norm{\chi^{n+1/2}}_{2,\phi} \\+& \norm{\xi^{n+1/2}}_{L^2(\mathbb{R})}\norm{\chi^{n+1/2}}_{2,
        \phi} \Bigg].
    \end{split}
\end{equation*} 
Using stability estimate \eqref{Stab} from the Lemma \ref{Lemma:H_alpha_bnd} along with the Cauchy's inequality, properties of $\phi$ and also dropping the second term on the left hand side, we obtain
\begin{align*}
    \norm{\chi^{n+1}}^2_{2,\phi} \leq& \norm{\chi^{n}}^2_{2,\phi} +\Delta t C(u_0,R)  \Big[ \norm{\chi^{n+1}}^2_{2,\phi} + \norm{\chi^{n}}^2_{2,\phi} + \norm{\rho^{n+1/2}}^2_{L^2(\R)} + |\rho^{n+1/2}|^2_{H^1(\R)} \\
    & \qquad+ |\rho^{n+1/2}|_{H^{2}(\R)}^2  + |\rho^{n+1/2}|_{H^{3}(\R)}^2 + \norm{\sigma^{n+1/2}}_{H^{3}(\R)}^2 + \norm{\xi^{n+1/2}}^2_{L^2(\R)} \Big],
    \end{align*}
which further implies 
 \begin{equation}\label{temp:rate_5}
(1-\Delta t C(u_0,R))\norm{\chi^{n+1}}_{2,\phi}^2 \leq (1+\Delta t C(u,R))   \norm{\chi^{n}}_{2,\phi}^2  + \Delta t C(u,R)S_n,
\end{equation}
where $S_n$ is represented by
\begin{equation*}
\begin{split}
    S_n := \norm{\rho^{n+1/2}}^2_{L^2(\R)} &+ |\rho^{n+1/2}|^2_{H^1(\R)}+ |\rho^{n+1/2}|^2_{H^2(\R)}\\& +|\rho^{n+1/2}|_{H^{3}(\R)}^2 + \norm{\sigma^{n+1/2}}_{H^{3}(\R)}^2 + \norm{\xi^{n+1/2}}^2_{L^2(\R)}.
    \end{split}
\end{equation*}
We choose $\Delta t$ in such a way that $(1-\Delta t C(u,R))\geq 1/2$. Afterwards, using the Taylor's formula with integral remainder, we have the following estimates
\begin{align}
    |\sigma^{n+1/2}|^2_{H^k(\R)} & \leq C\Delta t^3 \int_{t_n}^{t_{n+1}} | u_{tt}(s)|_{H^k(\R)}^2 \,ds,
   \label{estsemi}\\
    \norm{\xi^{n+1/2}}^2_{L^2(\R)} & \leq C\Delta t^3 \int_{t_n}^{t_{n+1}} \norm{u_{ttt}(s)}_{L^2(\R)}^2 \,ds.\label{estL2}
\end{align}
Using the estimates \eqref{POlyap}, \eqref{estsemi} and \eqref{estL2}, we obtain
\begin{equation*}
\begin{split}   
    S_n & \leq 2C\Delta x^{2(r+2)} \left(|u(t_n)|^2_{H^{r+2}(\R)} + |u(t_{n+1})|^2_{H^{r+2}(\R)}\right)\\& \quad + 2C\Delta x^{2(r+2)} \left(|u(t_n)|^2_{H^{r+2}(\R)} +|u(t_{n+1})|^2_{H^{r+2}(\R)}\right) \\&
    \quad + 2C\Delta x^{2r} \left(|u(t_n)|^2_{H^{r+2}(\R)} + |u(t_{n+1})|^2_{H^{r+2}(\R)}\right) \\&\quad+ 2C\Delta x^{2(r-1)}\left(|u(t_n)|^2_{H^{r+2}(\R)} +|u(t_{n+1})|^2_{H^{r+2}(\R)}\right) \\&
    \quad+  C\Delta t^3 \int_{t_n}^{t_{n+1}} \norm{u_{tt}(s)}_{H^{3}(\R)}^2 \,ds
    +C\Delta t^3 \int_{t_n}^{t_{n+1}} \norm{u_{ttt}(s)}_{L^2(\R)}^2 \,ds,\\ 
    &\leq C\Delta x^{2(r-1)} \sup_{0\leq t\leq T} |u(t)|^2_{H^{r+2}(\R)}\\ &\quad+ C\Delta t^3 \left(\int_{t_n}^{t_{n+1}} \norm{u_{tt}(s)}_{H^{3}(\R)}^2 \,ds+\int_{t_n}^{t_{n+1}} \norm{u_{ttt}(s)}_{L^2(\R)}^2\,ds\right).
\end{split} 
\end{equation*}
As a consequence, \eqref{temp:rate_5} becomes
\begin{equation*}
        \begin{split}
            \norm{\chi^n}^2_{2,\phi} &\leq \left(\frac{1+C\Delta t}{1-C\Delta t}\right)^n \|\chi^0\|_{2,\phi}^2 + \Delta t C \sum_{j=0}^{n-1}\left(\frac{1+C\Delta t}{1-C\Delta t}\right)^{n-j}S_j\\
            & \leq e^{4CT}\|\chi^0\|_{2,\phi}^2 + \Delta t e^{4CT}\sum_{j=0}^{n-1}S_j\\
            &\leq TC(u,R,T)\Delta x^{2(r-1)} + C(T) \Delta t^4 \left(\int_{0}^{T} \norm{u_{tt}(s)}_{H^{3}(\R)}^2 \,ds+\int_{0}^{T} \norm{u_{ttt}(s)}_{L^2(\R)}^2 \,ds\right)\\
            & = C(u,R,T) (\Delta x^{2(r-1)} + \Delta t^4).
        \end{split}
\end{equation*}
In order to ensure the bounded-ness of the above norms, we have taken into account $u_0\in H^s(\mathbb{R})$, where $s = \max\{r+2,5\}$.
We observe that using the estimate \eqref{POlyap}, for $n=1,2,\dots,N$
\begin{align*}
    \|u^n - u(t_n)\|_{L^2(\R)} & \leq \norm{\chi^n}_{L^2(\R)} + \|\rho^n\|_{L^2(\R)}\\
    &\leq \norm{\chi^n}_{2,\phi}+ \|\rho^n\|_{L^2(\R)}
    \leq C(u,R,T)(\Delta x^{r-1}+\Delta t^2).
\end{align*}
Hence the result follows.
\end{proof}
%%%%%%%%%%%%%% Numerical Section %%%%%%%%%%%%%%%%%%%%%%%%%%%%%%%
 \section{Numerical experiments}\label{sec6}
 %Here, we demonstrate several numerical examples of the fully discrete scheme \eqref{CNscheme}. We want to be clear that this is not an elaboration of a numerical method for the problem on the real line in the conventional sense, as these typically rely on applying a numerical scheme for the periodic version of the problem on a sufficiently large domain such that the reference solutions are close to zero outside of it, see, for example, the experimental sections of \cite{Vmurty,RajibBO}. We consider such an approach to be irrelevant to the primary conclusion reported here because this study solely addresses the convergence of the discretized real line problem. In our example, we will also take into account a discretized domain that is sufficiently large for the reference solutions to be almost zero outside of it. However, we will use the full line Hilbert transform, in which case we will simply ignore the contributions from outside the domain because the reference solution is almost zero there. By adding a periodic boundary condition that requires the approximation to take the same value at both ends of the domain, we are still able to approach our system of equations. The description of our presentation as follows numerically:
In our analysis, we provide a series of numerical illustrations of the fully discrete scheme \eqref{CNscheme} associated to \eqref{fkdv}. The conventional approaches typically involve applying a numerical scheme to the periodic version of the problem, considering a sufficiently large domain where the reference solutions tend to zero outside of it, for instance, kindly refer to \cite{ RajibBO, MonteCarlo, Vmurty}.
However, in particular, our study in this paper focuses on the convergence of the approximated solution on the real line. To address this, we consider a discretized domain that is large enough for the reference solutions (exact or higher-grid solutions) to be nearly zero outside of it. By incorporating a periodic boundary condition that enforces the approximation to have the same value at both ends of the domain, we are able to effectively handle our system of equations.

To summarize, we present a numerical demonstration according to the following procedure. Inspired by \cite{Galtungc,CNKDV}, we define the finite element space $S_{\Delta x}$. For this, let us consider the functions $f$ and $g$ as
 \begin{align*}
     f(y) =
     \begin{cases}
        1+y^2(2|y|-3), &|y|\leq 1,\\
        0,  &|y|>1,
     \end{cases} 
     \qquad g(y) = 
     \begin{cases}
        y(1-|y|)^2, & |y|\leq 1,\\
        0, &|y|>1.
     \end{cases}
 \end{align*}
 For $j\in \mathbb{Z}$, we consider $x_j=j\Delta x$ and define the basis functions as
 \begin{align*}
     v_{2j}(x) = f\left(\frac{x-x_j}{\Delta x}\right),\qquad v_{2j+1}(x) = g\left(\frac{x-x_j}{\Delta x}\right).
 \end{align*}
Then we consider $S_{\Delta x} = \text{span}(\{v_{2j},v_{2j+1}\}: j=-M,-M+1,\dots,M)$ which is a $4M+2$ dimensional subspace of $H^{1+\alpha}(\mathbb{R})$ for $\alpha\in[1,2)$. 
 %In the following we define $N:=2M$, which is the number of elements used in the numerical experiment. To ensure that the weight function $\phi$ is a positive, increasing function whose derivative equals 1 on the domain $[-100, 100]$ under consideration, we have set its value to $120 + x$. 
 
 %Smaller time increments did not significantly increase the convergence rates of the approximations, thus we decided to select $\Delta t = \mathcal{O}(\Delta x)$  in our experiments in contrast to the theory's claim that $\Delta t = \mathcal{O}(\Delta x^2)$. The stopping condition $\norm{w^{\ell+1}-w^\ell}_{L^2(\R)}\leq0.002\Delta x\norm{u^n}_{L^2(\R)}$, which normally needed 4–7 iterations for the coarser discretizations and 2-3 iterations for the finer ones, was chosen in iteration \eqref{solv} to get $u^{n+1}$.On each element, $\langle \mathcal{D}^{\alpha}(v_j), (v_i)\rangle$ was determined by first using an eight point Gauss-Legendre (GL) quadrature rule on the inner product of the principal value integral defining $\mathcal{D}^{\alpha}(v_j)$ in seven GL points.
Let us denote $N:=2M$ as the number of elements considered for our numerical experiments. As mentioned in \cite{CNKDV} for KdV equation and \cite{Galtungc} for BO equation, setting $\Delta t = \mathcal{O}(\Delta x)$ is sufficient to obtain the numerical results instead of $\Delta t = \mathcal{O}(\Delta x^2)$. To determine the subsequent iteration $u^{n+1}$ in equation \eqref{solv}, we impose a termination criterion asserting that $\norm{w^{\ell+1}-w^\ell}_{L^2(\R)}\leq0.002\Delta x\norm{u^n}_{L^2(\R)}$.

For the computation of the inner product $\langle \mathcal{D}^{\alpha}(v_j), (v_i)\rangle$ for each individual element, we use an eight-point \emph{Gauss-Legendre} (GL) quadrature rule. This rule was applied to the principal value integral defining $\mathcal{D}^{\alpha}(v_j)$ and is evaluated at seven GL points. We have measured the relative $L^2$-error, which is defined by 
\begin{equation*}
     E:= \frac{\|u_{\Delta x}-u\|_{L^2(\R)}}{\|u\|_{L^2(\R)}},
 \end{equation*}
 where the $L^2$-norms are computed using the trapezoidal rule on the points $x_j$. For $t=n\Delta t$, we set
 $$u_{\Delta x}(x,t) = u^n(x,t) = \sum_{j=-2M}^{2M+1}u_j^nv_j(x).$$

Similar to the case of usual KdV equation, the fractional KdV equation \eqref{fkdv} possesses an infinite number of conserved quantities \cite{Galtungc,FKDV_WP_L2,HOLDEN1999}. Hereby we consider the first two specific quantities known as $mass$ and $momentum$ along with another conserved quantity involving the exponent $\alpha$. With normalization, these quantities can be expressed as follows:
\begin{align*}
     C^{\Delta}_1 &:= \frac{\int_{\mathbb{R}} u_{\Delta x}\,dx}{\int_{\mathbb{R}} u^0\,dx},\qquad
     C^{\Delta}_2 := \frac{\|u_{\Delta x}\|_{L^2(\R)}}{\norm{u^0}_{L^2(\R)}},\\
     C^{\Delta}_3:&= \frac{\int_{\R} \left((\mathcal{D}^{\alpha/2}u_{\Delta x})^2 - \frac{(u_{\Delta x})^3}{3}\right)~dx}{\int_{\R} \left((\mathcal{D}^{\alpha/2}u^0)^2 - \frac{(u^0)^3}{3}\right)~dx}.
\end{align*}
Our aim is to preserve these quantities in the discrete set up. 
We remark that within the realm of completely integrable partial differential equations, it has been noted that numerical methodologies capable of preserving a greater number of conserved quantities generally yield more accurate approximations in comparison to those preserving fewer. 
Furthermore, we have determined the convergence rates of the numerical scheme \eqref{CNscheme}, denoted as $R_E$, with the varying numbers of elements $N_1$ and $N_2$, defined by
\begin{equation*}
     \frac{\ln(E(N_1))-\ln(E(N_2))}{\ln(N_2)-\ln(N_1)},
 \end{equation*}
where $E$ is considered as a function of the number of elements $N$.
 %%%%%%%%%%%%%%%%%%%%%%%%%%%%%%%%%%%%%%%%%%%%%%%%%%%%%%%%%%%%%%%%
 \subsection{Benjamin Ono equation} ($\alpha\approx1$):
  In this example we consider the one soliton solution of the Benjamin Ono equation presented in \cite{Vmurty}, namely

\begin{equation}
    u(x,t) = \frac{2c\delta}{1-\sqrt{1-\delta^2}\cos(c\delta(x-ct))}, \qquad \delta = \frac{\pi}{cL}.
\end{equation}
We have applied the proposed scheme \eqref{CNscheme} with the initial data $u_0(x) = u(x,0)$ along with the parameters $c=0.25$ and $L=15$. We set the time step to $\Delta t = \Delta x/\norm{u_0}_{\infty}$ and the approximate solution is computed at $t=120$ which is a period for the exact solution. 
\begin{table} 
\begin{tabularx}{0.8\textwidth} { 
   >{\raggedright\arraybackslash}X 
   >{\raggedright\arraybackslash}X
   >{\centering\arraybackslash}X 
   >{\raggedleft\arraybackslash}X
   >{\raggedleft\arraybackslash}X 
   >{\raggedleft\arraybackslash}X }
 \hline
 N & E & $C^{\Delta}_1$ & $C^{\Delta}_2$ & $C^{\Delta}_3$ & $R_E$ \\
 \hline
 \hline
 \\
 64  & 0.2898  & 1.19 &1.310 & 3.850\\
    &  & & & & 3.590 \\
 128 & 0.0241 & 1.108 & 1.275 & 3.055\\
     &  & & & & 2.462\\
 256  & 0.0044 & 1.000 & 1.016 & 1.034\\
  &  & & & & 2.044 \\
 512 & 0.0011 & 1.000 & 1.001 & 1.001\\ 
    &  & & & & 2.002\\
 1024  & 0.0002  & 1.000 & 1.000 & 1.000\\
 \\
\hline
\end{tabularx}
\caption{Relative error for one soliton solution with $\alpha=1$.}
\label{tab_1}
\end{table}
A visualization for the results for $N=128$ and $512$ are given in the Figure \ref{fig:BO1}. Clearly the plot indicates that the approximated solution converges to the exact solution and this is confirmed by the error analysis in Table \ref{tab_1}. It demonstrates the errors are small even for fairly coarser grids and it is converging to zero at an optimal rate $2$. 
\begin{figure}
\centering
\includegraphics[width=0.7\linewidth, height=7cm]{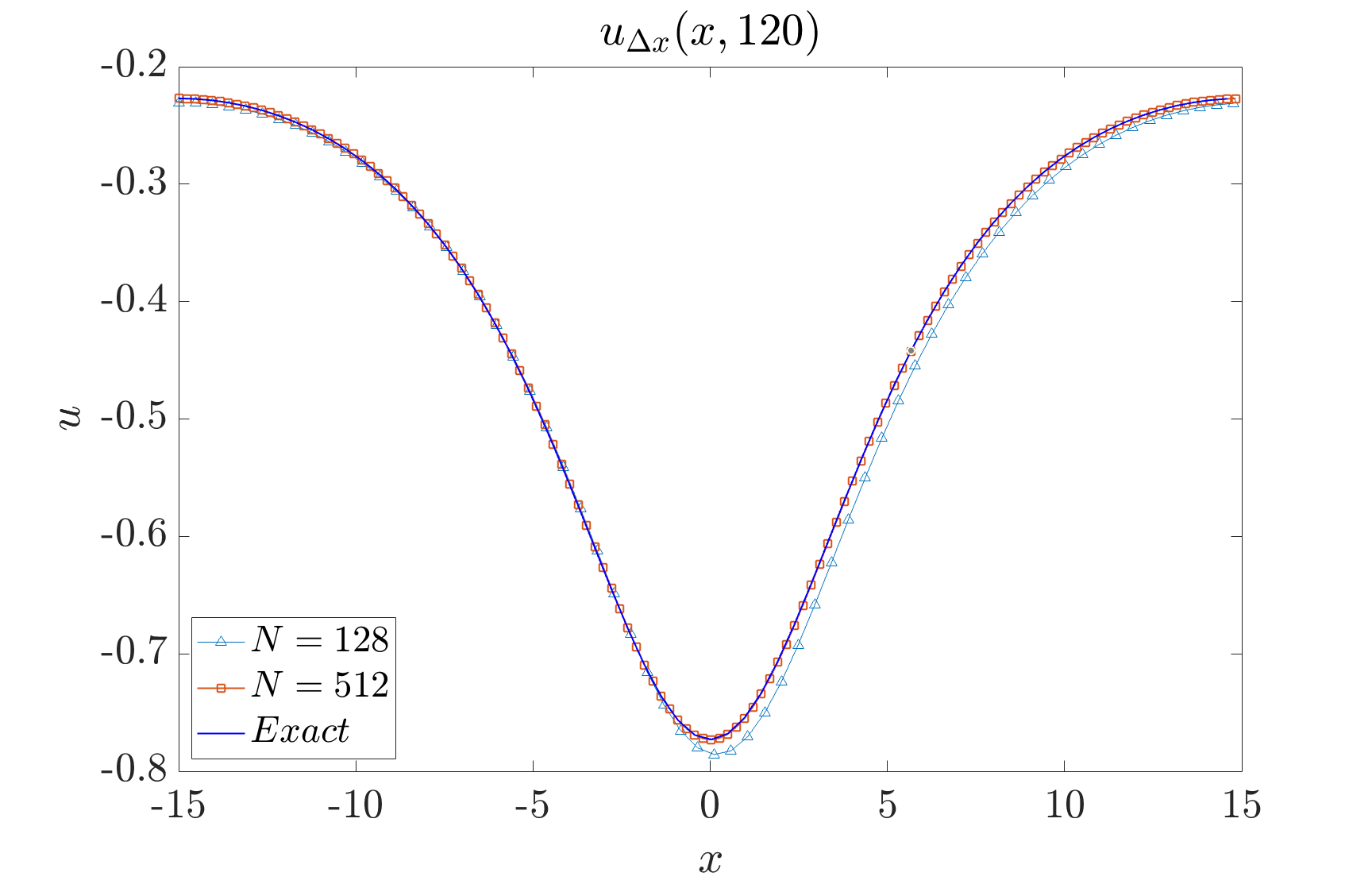}
\caption{The exact and numerical solution at $t=120$ with the initial data $u(x,0)$ using $N=128$ and $N=256$.}
\label{fig:BO1}
\end{figure}
 %\begin{equation}\label{BOSOln}
%     u_{s2}(x,t) = \frac{4c_1c_2(c_1\lambda_1^2+c_2\lambda_2^2+(c_1+c_2)^2c_1^{-1}c_2^{-1}(c_1-c_2)^{-2})}{(c_1c_2\lambda_1\lambda_2-(c_1+c_2)^2(c_1-c_2)^{-2})^2+(c_1\lambda_1+c_2\lambda_2)^2},
 %\end{equation}

 %where $\lambda_1 := x-c_1t-d_1$ and $\lambda_2 := x-c_2t-d_2$. When $c_2>c_1$ and $d1>d_2$ this expression represents a tall soliton overtaking a smaller one while moving to the right. we applied the fully discrete scheme with the following values:
% \begin{align*}
 %    Initial~ data: ~ u_0(x) &= u_{s2}(x,0)\\
 %    Parameters:~~    c_1 &= 0.3,\\
 %                     c_2 &= 0.6,\\
 %                     d_1 &= -30,\\
 %                     d_2 &= -55,\\
 %                     \Delta t &= 0.5\norm{u_0}^{-1}_{L^\infty}\Delta x,\\

     %Numerical~solution~to~be ~computed~at : t&=90 ~~and~ \\
  %                                 t&=180.\\
  %                           we~set~the~Domain:~x&\in[-100,100]\\
                  %Number~of~elements~N&=256~for~which~comparison~between~\\
%approximation~ and~exact~solution~ is~to~be~shown~in~Figure~1.
% \end{align*}
 \subsection{KdV equation}  ($\alpha \approx 2 $):
 We would like to compare our approximated solution $u_{\Delta x}$ obtained from \eqref{CNscheme} in case of $\alpha = 1.999 $ with the exact solution of the KdV equation \(u_t + (u^2/2)_x +u_{xxx} = 0\). We test our scheme \eqref{CNscheme} for one-soliton and two-soliton solution.
 \subsubsection{One soliton }
 The family of exact solution (one soliton) of the KdV equation is given by \cite{dutta2015convergence}
 \begin{equation}\label{onesol}
     u(x,t) = 9\left(1-\tanh^2\left(\sqrt{3/2}(x-3t)\right)\right),
\end{equation}
which represents a single `bump' moving to the right with speed $3$.
We have tested our scheme with the initial data $u_0 = u(x,-1)$. The solution of \eqref{fkdv} is calculated on the uniform grid with $\Delta x = 30/N$ in the interval $[-15,15]$. The Figure \ref{fig:KDVone} depicts the convergence of the approximated solution and the Table \ref{tab_2} provides the expected rate of convergence.
\begin{table}
\begin{tabularx}{0.8\textwidth} { 
   >{\raggedright\arraybackslash}X
    >{\raggedright\arraybackslash}X
   >{\centering\arraybackslash}X 
   >{\raggedleft\arraybackslash}X
   >{\raggedleft\arraybackslash}X 
   >{\raggedleft\arraybackslash}X }
 \hline
 N & E & $C^{\Delta}_1$ & $C^{\Delta}_2$ & $C^{\Delta}_3$ & $R_E$ \\
 \hline
 \hline 
 \\
 32  & 0.0693 &  1.10 & 1.000 & 1.629\\
   & & & &  & 1.717 \\
 64  & 0.0211  & 1.02 &1.000 & 1.353\\
  &  & & & & 1.922 \\
 128 & 0.0056 & 1.01 & 1.00 & 1.076  \\
   &  & & & & 1.955 \\
 256  & 0.0014 & 1.00 & 1.00 & 1.006\\
  &  & & & & 1.990\\
 512 & 3.5613e-04  & 1.00 & 1.00 & 1.005 \\
   &  & & & & 1.998 \\
 1024  & 8.9185e-05   & 1.00 & 1.00 & 1.001\\
  &  & & & &1.992\\
 2048  & 2.2423e-05  & 1.00 & 1.00  & 1.000\\
 \\
\hline
\end{tabularx}
\caption{Relative error for one soliton solution ($\alpha=1.999$).}
\label{tab_2}
\end{table}
%%%%%%%%
\begin{figure}
    \centering
    \includegraphics[width=0.8\linewidth, height=8cm]{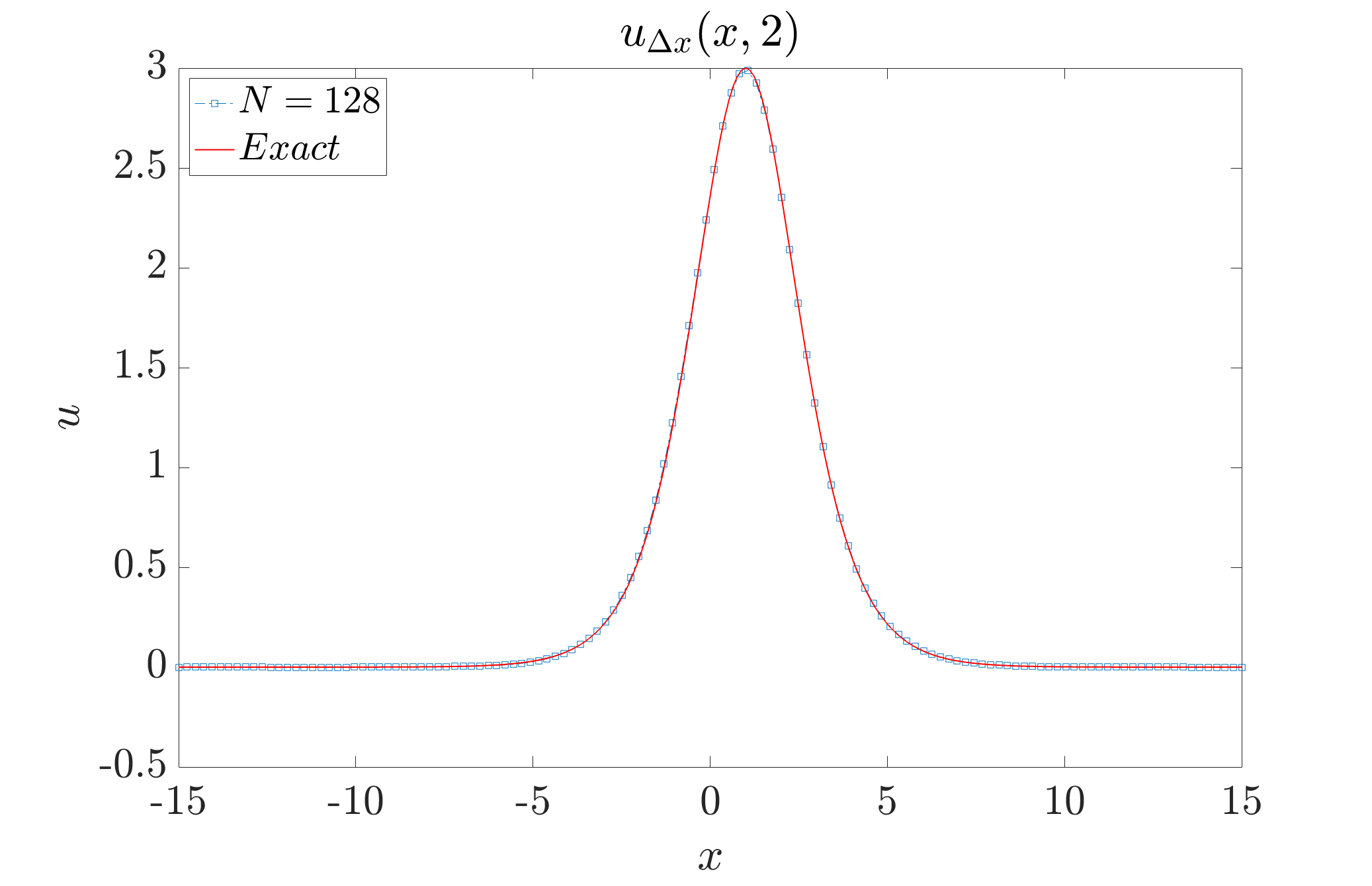}
    \caption{The exact and numerical solution at $t=2$ with the initial data $w_1(x,-1)$ with $N=128$.}
    \label{fig:KDVone}
\end{figure}
%%%%%%%%%%%%%%
\subsubsection{Two soliton solution}
From a physical perspective, solitons of different shapes manifest diverse velocities, establishing a correlation between soliton height and speed. A taller soliton exhibits a greater swiftness compared to a shorter one. As two solitons traverse a surface, the taller soliton surpasses the shorter soliton, and both solitons emerge unaltered after the collision. This scenario presents a considerably more intricate computational challenge than solving for a single soliton solution.
In case of two soliton, the family of exact solution of the KdV equation is given by \cite{dutta2015convergence}
 \begin{equation}\label{twosol}
     u(x,t) = 6(b-a) \frac{b \csch^2\left(\sqrt{b/2}(x-2bt)\right)+a \sech^2\left(\sqrt{a/2}(x-2at)\right)}{\left(\sqrt{a}\tanh\left(\sqrt{a/2}(x-2at)\right) - \sqrt{b} \coth\left(\sqrt{b/2}(x-2bt)\right)\right)^2}.
\end{equation}
for some constants $a$ and $b$. We have considered the parameters $a=0.5$ and $b=1$, and the initial data $u_0(x) = u(x,-10)$.
%Physically, two solitons which have different shapes move with different velocities, which is a dependence between the height of the soliton and the velocity. 
%A higher soliton moves more quickly than a lower soliton.
%If the two solitons travel over a surface, the higher soliton will overtake the lower soliton, and both solitons will emerge unchanged after the collision. 
%This is a notably more difficult problem to solve computationally than the one-soliton solution.
%At $t = 20$, we calculated the approximate solution. In this example, the exact solution is $w_2(x, 10)$. The exact and numerical solutions at t = 20 are shown in Figure .  This is owing to an errors in the bigger bump's location, which is caused by a much smaller error in the height of the bump. Because of this inaccuracy, the wave's speed is somewhat higher than the speed of the comparable wave in the exact solution. Because the wave is relatively small, the L2 error is rather big. The percentage errors for the two-soliton simulation, as well as an error in L2 conservation, are shown in Table below.
We computed the approximated solution at time $t = 20$ to compare with the exact solution $u(x, 10)$. The Figure \ref{fig:KdVtwo} represents the exact solution at $t=10$ and numerical solution at $t = 20$. Nevertheless, slight inaccuracies in the placement of the larger bump, stemming from a minor discrepancy in its height, result in a slightly higher velocity for the wave compared to the corresponding wave in the exact solution. Since the wave is relatively narrow, the $L^2$-error assumes significant proportions. The Table \ref{tab_3} demonstrates the relative $L^2$-errors for the two soliton simulation.
\begin{table}
\begin{tabularx}{0.8\textwidth} { 
  >{\raggedright\arraybackslash}X 
  >{\raggedright\arraybackslash}X
   >{\centering\arraybackslash}X 
   >{\raggedleft\arraybackslash}X
   >{\raggedleft\arraybackslash}X
   >{\raggedleft\arraybackslash}X }
 \hline
 N & E  & $C^{\Delta}_1$ & $C^{\Delta}_2$ & $C^{\Delta}_3$ & $R_E$\\
 \hline
 \hline
 \\
 256  & 1.098  & 0.978 & 0.770 & 26.518  \\
  &  &  &  & & 0.57 \\
 512 & 0.741  & 0.988 & 0.889 & 10.788  \\
   &  &  &  & & 1.53 \\
 1024  & 0.257 & 1.001 & 0.934 &3.826\\
  &  &  &  &  &1.88\\
 2048 & 0.070 & 1.001 & 0.955 & 1.722\\
   &  &  &  & &1.98\\
 4096 & 0.018  & 1.000 & 0.969 & 1.042 \\
  &  &  &  &  &\\
\hline
\label{TableTwoKdV}
\end{tabularx}
\caption{Relative error for two soliton solutions ($\alpha = 1.999$)}
\label{tab_3}
\end{table}
\begin{figure}
    \centering
    \includegraphics[width=0.8\linewidth, height=8cm]{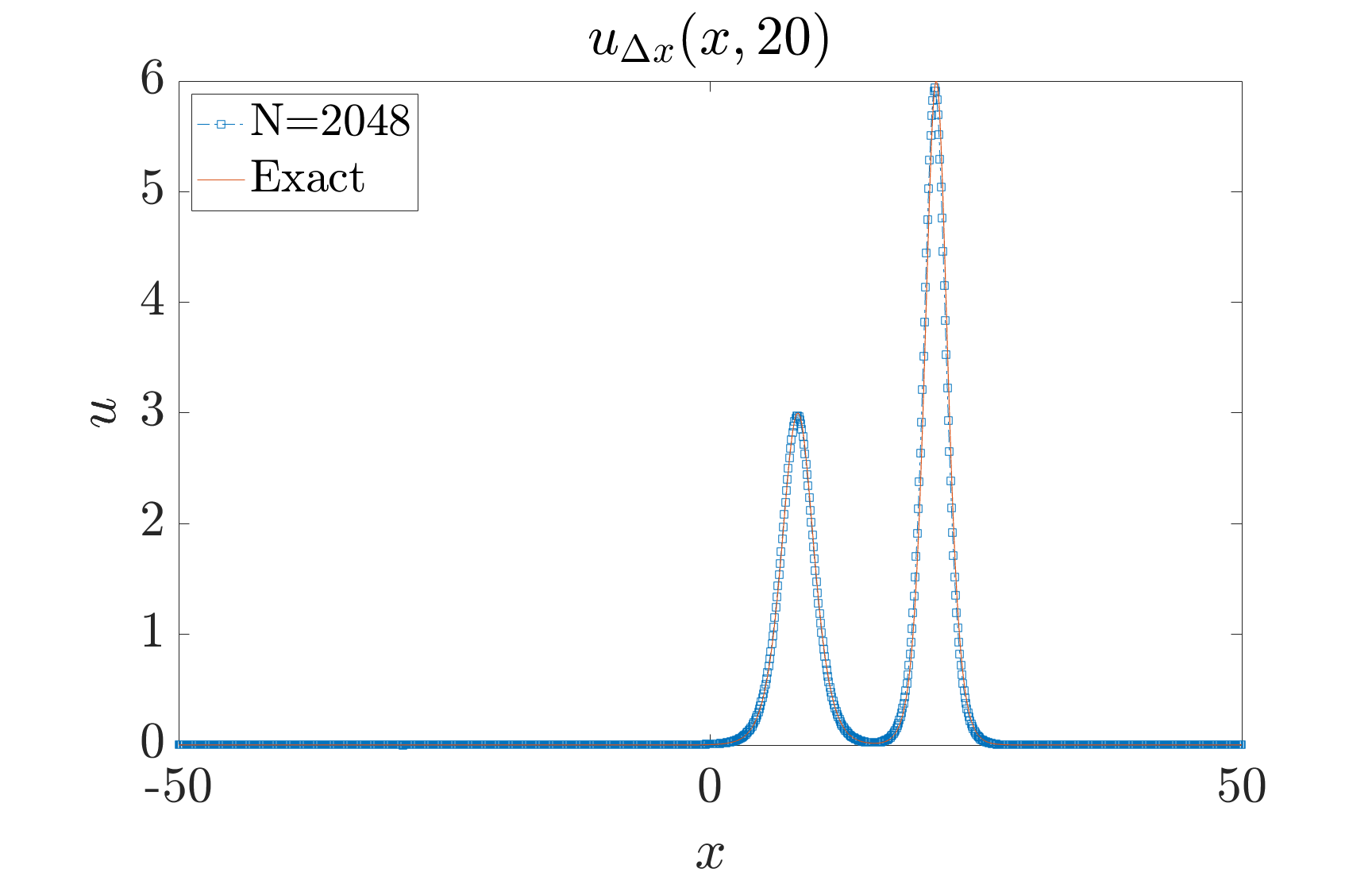}
    \caption{The exact and numerical solution at $t=20$ with the initial data $w_2(x,-10)$ with $N=2048$ $(\alpha=1.99)$}
    \label{fig:KdVtwo}
\end{figure}
\subsection{Fractional case with smooth initial data}($\alpha=1.5$) We test the convergence for the initial condition $u_0(x) = 0.5 \sin(x)$ in the interval $[0,2\pi]$. In the Table \ref{tab_4}, we have analyzed the $L^2$ errors using the approximated solution with $2^{16}$ grid points as a reference solution at time $T = 1$. The relative $L^2$-errors along with the conserved quantities converge at the optimal rates as mesh size decreases.
\begin{figure}
    \centering
    \includegraphics[width=0.8\linewidth, height=8cm]{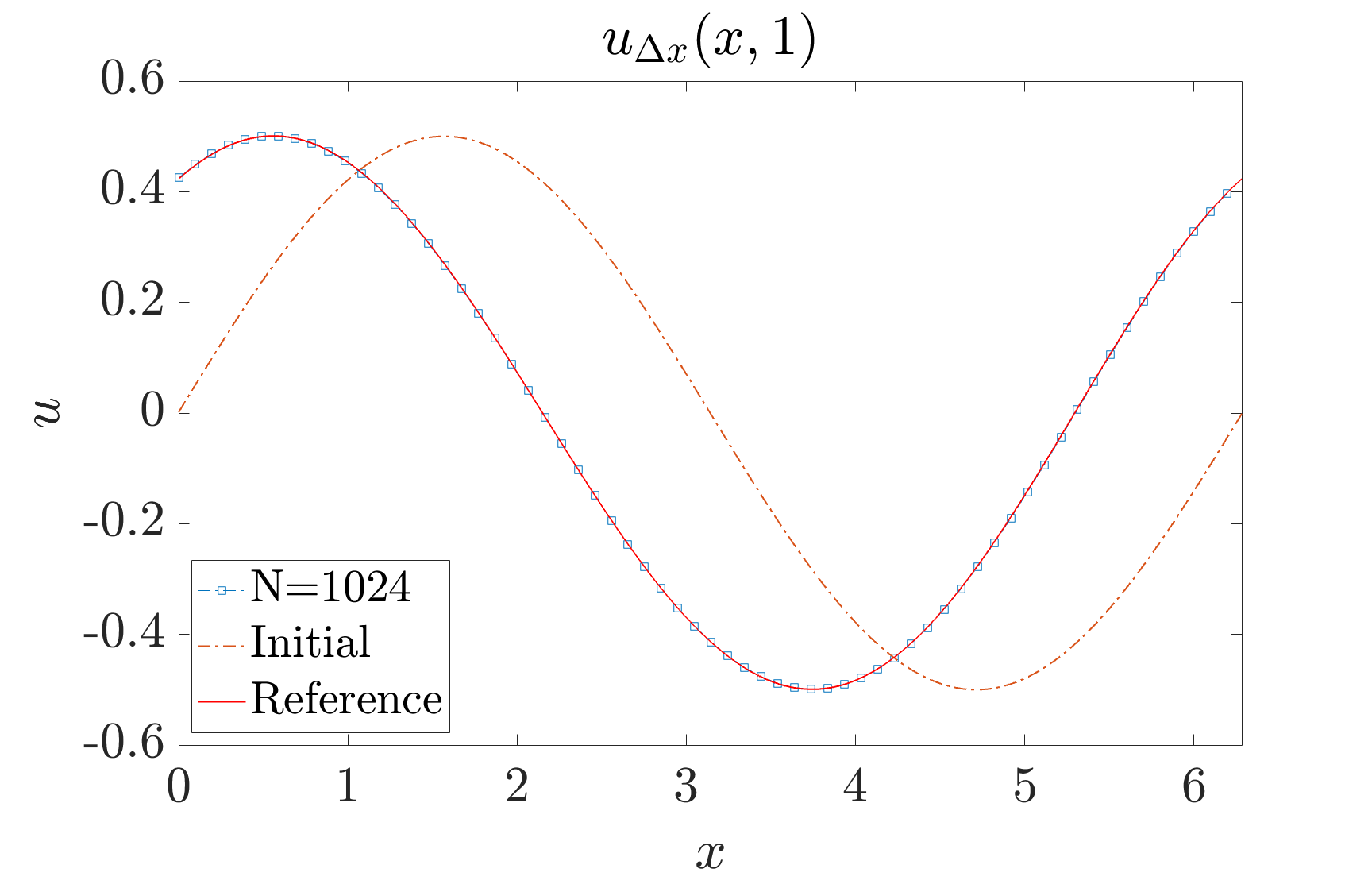}
    \caption{The exact and numerical solution at $t=1$ with the initial data $u_0(x) = 0.5 \sin(x)$ with $N=1024$ grids. $\alpha=1.5$}
    \label{fig:alpha1_5smooth}
\end{figure}
The Figure \ref{fig:alpha1_5smooth} describes that the approximated solution converges to the reference solution even for the fairly coarser grids for the parameter $\alpha=1.5$. 
\begin{table}
\begin{tabularx}{0.8\textwidth} { 
  >{\raggedright\arraybackslash}X 
  >{\raggedright\arraybackslash}X
   >{\centering\arraybackslash}X 
   >{\raggedleft\arraybackslash}X
   >{\raggedleft\arraybackslash}X
   >{\raggedleft\arraybackslash}X }
 \hline
 N & E  & $C^{\Delta}_1$ & $C^{\Delta}_2$ & $C^{\Delta}_3$ & $R_E$\\
 \hline
 \hline
\\
 512  & 0.0011 & 1.001 & 1.000 & 1.001 \\
  &  &  &  & & 1.984\\
 1024 & 0.0003 & 1.001 & 1.000 & 1.001\\
   &  &  &  & & 1.980\\
 2048 & 6.822e-05 & 1.001 & 1.000 & 1.000 \\
  &  &  &  & & 1.994 \\
  4096 & 1.7128e-05  & 1.000 & 0.998 & 1.000 \\
  &  &  & & &2.005\\
  8192 & 4.2688e-06 & 1.000 & 0.999 & 1.000\\
  &  &  &  &  &1.888 \\
  16384 & 1.1531e-06 & 1.000 & 1.000 & 1.000\\
  &  &  &  &  & \\
\hline

\label{Tablealpha15smooth}
\end{tabularx}
\caption{Relative errors for $\alpha =1.5$ with smooth initial data.}
\label{tab_4}
\end{table}
\subsection{Fractional case with non-smooth initial data} ($\alpha=1.5$):
In our final illustration, we test the convergence of our scheme \eqref{CNscheme} for the less regular initial data 
\begin{equation*}
u_0(x) =
    \begin{cases}
         \frac{1}{2} (x+1), & \qquad x\in[-1,1],\\
                  0,& \qquad \text{otherwise}
    \end{cases}
\end{equation*}
in the interval $[-10,10]$. The Figure \ref{fig:alpha1_5L2} depicts the approximated solution for various number of elements. In the Table \ref{tab_5}, we analyze the $L^2$ errors using the approximated solution obtained by \eqref{CNscheme} for the exponent $\alpha=1.5$ with $2^{16}$ grid points as a reference solution at time $T = 0.1$ due to the unavailability of the reliable reference solution. We observe that the relative $L^2$-error is decreasing. However, the large errors and slow convergence in Table \ref{tab_5} indicates that we are not yet in the asymptotic regime. 
\begin{table}
\begin{tabularx}{0.8\textwidth} { 
  >{\raggedright\arraybackslash}X 
  >{\raggedright\arraybackslash}X 
   >{\centering\arraybackslash}X 
   >{\raggedleft\arraybackslash}X
   >{\raggedleft\arraybackslash}X
   >{\raggedleft\arraybackslash}X }
 \hline
 N & E  & $C^{\Delta}_1$ & $C^{\Delta}_2$ & $C^{\Delta}_3$ & $R_E$\\
 \hline
 \hline
 \\
 2048 & 0.4482 & 0.05 & 0.176 & 0.022 \\
   &  &  &  & & 0.192\\
 4096 & 0.3925 & 0.12 & 0.250 & 0.043\\
  &  &  &  & & 0.039 \\
  8192 & 0.3820 & 0.26 & 0.354 & 0.084 \\
  &  &  &  & & 0.162 \\
  16384 & 0.3416 & 0.53 &  0.500 & 0.173\\
  &  &  &  &  & 0.117 \\
  32768 & 0.3150 & 1.08 & 0.707 & 0.333\\
  &  &  &   \\
\hline
\end{tabularx}
\caption{Relative errors for $\alpha =1.5$ with $L^2$ initial data.}
\label{tab_5}
\end{table}
\begin{figure}
    \centering
    \includegraphics[width=0.9\linewidth, height=9cm]{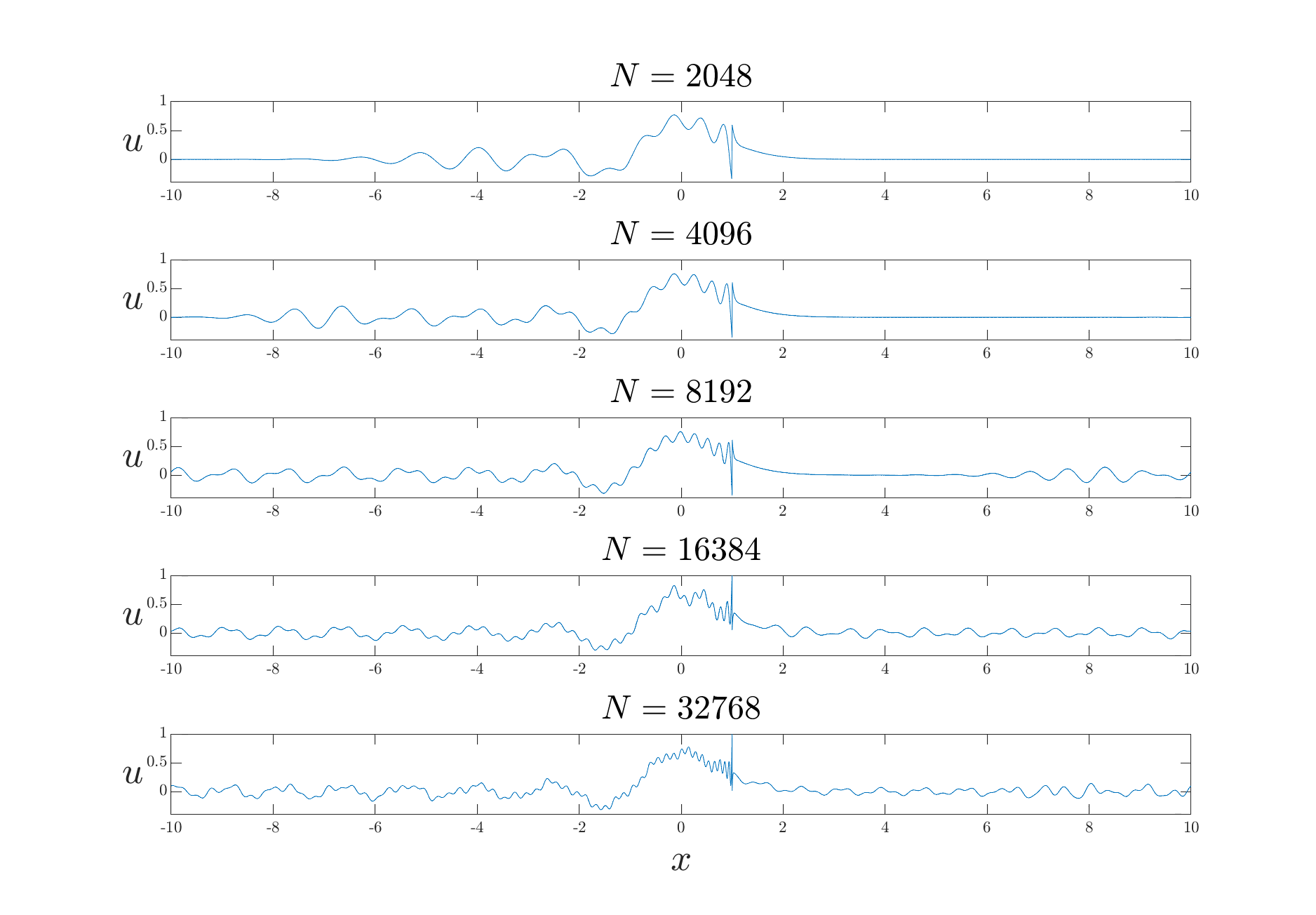}
    \caption{The approximated solution at $t=0.1$ with the $L^2$ initial data $u_0(x)$.}
    \label{fig:alpha1_5L2}
\end{figure}

We have observed that the approximated solutions of the equation \eqref{fkdv}, obtained using the devised scheme \eqref{CNscheme} demonstrates the convergence towards the exact solutions for several values of $\alpha$ numerically. This convergence exhibits an expected rate of $2$, which is consistent with our theoretical results as demonstrated in the earlier sections.
It is important to note that when dealing with non-smooth solutions, the convergence rates do not adhere to the expected behavior, and this inconsistency can be attributed to some of the evident factors, such as reference solution may not be close to the exact solution. However, even though the data is less regular, the conserved quantities are converging but a lower convergence rate. This is expected as we are not yet in an asymptotic regime.

\subsection*{Acknowledgements} 
TS is partially supported by the MATRICS Grant MTR/2021/000810 from the
Science and Engineering Research Board, Government of India.
%%%%%%%%%%%%%%%%%%%%%%%%%%%%%%%%%%%


\begin{thebibliography}{99}
	
\bibitem{Bonaf}
L.  Abdelouhab, J. L.  Bona, M. Felland, and J. -C. Saut.
\newblock{Nonlocal models for nonlinear, dispersive waves.}
\newblock{\em Physica D: Nonlinear Phenomena},  40 (1989), no. 3, 360--392.
	
\bibitem{biccari}
U. Biccari, M. Warma, and E. Zuazua.
\newblock{Local elliptic regularity for the Dirichlet fractional Laplacian.}
\newblock{\em Advanced Nonlinear Studies,}  17 (2017), no. 2, 387--409.

\bibitem{Ciarlet}
P. G. Ciarlet.
\newblock{The Finite Element Method for Elliptic Problems.}
\newblock{\em Society for Industrial and Applied Mathematics},  2002.
	
\bibitem{Hitchhiker}
E. Di Nezza, G. Palatucci, and E. Valdinoci.
\newblock{Hitchhiker's guide to the fractional Sobolev spaces}.
\newblock{\em Bulletin des sciences math{\'e}matiques}, 136 (2012), no. 5, 521--573.
	
\bibitem{RajibBO}
R. Dutta, H. Holden, U. Koley, and N. H. Risebro.
\newblock{Convergence of finite difference schemes for the Benjamin--Ono equation}.
\newblock{\em Numerische Mathematik}, 134 (2016), no. 2, 249--274.
	
\bibitem{dutta2015convergence}
R. Dutta, U. Koley, and N. H. Risebro.
\newblock{Convergence of a higher order scheme for the Korteweg–de Vries equation}.
\newblock{\em SIAM Journal on Numerical Analysis}, 53 (2015), no. 4, 1963--1983.
	
\bibitem{RDOS}
R. Dutta, H. Holden, U. Koley, and N. H. Risebro.
\newblock{Operator splitting for the Benjamin–Ono equation.}
\newblock{\em Journal of Differential Equations}, 259 (2015), no. 11, 6694--6717.

\bibitem{CNKDV}
R. Dutta, and N. H. Risebro.
\newblock{ A note on the convergence of a Crank–Nicolson scheme for the KdV equation}.
\newblock{\em Int. J. Numer. Anal. Model}, 13 (2016), no. 5, 657--675.

\bibitem{TanmayS}
R. Dutta, and T. Sarkar.
\newblock{Operator splitting for the fractional Korteweg-de Vries equation}.
\newblock{\em Numerical Methods for Partial Differential Equations}, 37 (2021), no. 6, 3000--3022.

 \bibitem{BO_ponce}
G. Fonseca, F. Linares, and G. Ponce.
\newblock{ The IVP for the dispersion generalized Benjamin--Ono equation in weighted Sobolev spaces}.
\newblock{\em  Annales de l'Institut Henri Poincar{\'e} C, Analyse non lin{\'e}aire}, 30 (2013), no. 5, 763--790.

\bibitem{CNGALTUNG}
S. T. Galtung.
\newblock{Convergent Crank–Nicolson Galerkin Scheme for the Benjamin–Ono Equation, Master’s thesis,}
\newblock{\em NTNU Norwegian University of Science and Technology},(2016) 
 
\bibitem{CRGALTUNG}
S. T. Galtung.
\newblock{Convergence rates of a fully discrete Galerkin scheme for the Benjamin–Ono equation,}
\newblock{\em XVI International Conference on Hyperbolic Problems: Theory, Numerics, Applications}, Springer, (2016), 589--601.

\bibitem{Galtungc}
S. T. Galtung.
\newblock{A convergent Crank–Nicolson Galerkin scheme for the Benjamin–Ono equation}.
\newblock{\em Discrete and Continuous Dynamical Systems}, 38 (2018), no. 3, 1243--1268.

\bibitem{VeloC}
J. Ginibre, and G. Velo.
\newblock{Commutator expansions and smoothing properties of generalized Benjamin-Ono equations}.
\newblock{\em Annales de l'IHP Physique th{\'e}orique}, 51 (1989), no. 2, 221--229.
 
\bibitem{VeloS}
J. Ginibre, and G. Velo.
\newblock{Smoothing properties and existence of solutions for the generalized Benjamin–Ono equation}.
\newblock{\em Journal of differential equations}, 93 (1991), no. 1, 150--212.

\bibitem{Grubb}
G. Grubb.
\newblock{Fourier methods for fractional-order operators}.
\newblock{\em arXiv preprint arXiv:2208.07175}, (2022).
	
\bibitem{FKDV_WP_L2}
S. Herr, A. D. Ionescu, C. E. Kenig, and H. Koch.
	\newblock{ A para-differential renormalization technique for nonlinear dispersive equations}.
	\newblock{\em Communications in Partial Differential Equations}, 35 (2010), no.10, 1827--1875.

\bibitem{HOLDEN1999}
	H. Holden, K. H. Karlsen, and N. H. Risebro.
	\newblock{Operator Splitting Methods for Generalized Korteweg–De Vries Equations.}
	\newblock{\em Journal of Computational Physics}, 153 (1999), no. 1, 203--222.
 
\bibitem{FullydiscreteKDV}
	H. Holden, U. Koley, and N. H. Risebro.
	\newblock{Convergence of a fully discrete finite difference scheme for the Korteweg–de Vries equation.}
	\newblock{\em IMA Journal of Numerical Analysis}, 35 (2015), no. 3, 1047--1077.

\bibitem{Kenig_wp_BO}
	A. D. Ionescu and C. E. Kenig.
	\newblock{Global well-posedness of the Benjamin–Ono equation in low-regularity spaces.}
	\newblock{\em J. Amer. Math. Soc.}, 20 (2007), 753-798.
 
\bibitem{podlubny1998fractional}
	 Igor Podlubny.
	\newblock{Fractional differential equations: an introduction to fractional derivatives, fractional differential equations, to methods of their solution and some of their applications.}
	\newblock{\em Elsevier}, (1998).

 
\bibitem{KATO}
	T. Kato.
	\newblock{ On the Cauchy problem for the (generalized) Korteweg–de Vries equation.}
	\newblock{\em Studies in Appl. Math. Ad. in Math. Suppl. Stud.}, (1983), no. 8, 93--128.
	
\bibitem{Kenig}
	C. E. Kenig, G. Ponce, L. Vega.
	\newblock{Well-Posedness of the Initial Value Problem for the Korteweg-de Vries Equation.}
	\newblock{\em Journal of the American Mathematical Society}, 4 (1991), no. 2, 323--347.
	
 \bibitem{Visan}
	 R. Killip and M. Vi{\c{s}}an.
	\newblock{KdV is wellposed in $H^{-1}$.}
	\newblock{\em Annals of Mathematics}, 190 (2019), no. 1, 249--305.
	
\bibitem{Saut}
	 C. Klein, and J.-C. Saut.
	\newblock{A numerical approach to blow-up issues for dispersive perturbations of Burgers' equation}.
	\newblock{\em Physica D: Nonlinear Phenomena}, 295 (2015), 46--65.

\bibitem{MonteCarlo}
U. Koley, D. Ray, and T. Sarkar.
\newblock{Multilevel Monte Carlo Finite Difference Methods for Fractional Conservation Laws with Random Data}.
\newblock{\em SIAM/ASA Journal on Uncertainty Quantification}, 9 (2021), no.1, 65--105.

\bibitem{KdV}
	 D. J.   Korteweg,  and  G.   De Vries .
	\newblock{XLI. On the change of form of long waves advancing in a rectangular canal, and on a new type of long stationary waves}.
	\newblock{\em The London, Edinburgh, and Dublin Philosophical Magazine and Journal of Science}, 39 (1895), no. 240, 422--443.

 \bibitem{Pilod}
	F. Linares, D. Pilod, and J.-C. Saut.
	\newblock{ Dispersive perturbations of Burgers and hyperbolic equations I: Local theory}.
	\newblock{\em SIAM Journal on Mathematical Analysis},  46 (2014), no. 2, 1505--1537.

\bibitem{Ponce}
	F. Linares, and G. Ponce.
	\newblock{Introduction to Nonlinear Dispersive Equations}.
	\newblock{\em University text, Springer}, 2014.

\bibitem{10Definition}
M. Kwa{\'s}nicki.
	\newblock{Ten equivalent definitions of the fractional Laplace operator}.
	\newblock{\em Fractional Calculus and Applied Analysis}, 20 (2017), no. 1, 7--51.

\bibitem{molinet2012cauchy}
L. Molinet, and D. Pilod.
\newblock{The Cauchy problem for the Benjamin--Ono equation in $L^2$ revisited}.
\newblock{\em Analysis \& PDE}, 5 (2012), no. 2, 365--395.
	
\bibitem{Vento}
	 L. Molinet, D. Pilod, and S. Vento.
	\newblock{On well-posedness for some dispersive perturbations of
Burgers' equation}.
	\newblock{\em Annales de l'Institut Henri Poincar{\'e} C, Analyse non lin{\'e}aire}, 35 (2018), no. 7, 1719--1756.

\bibitem{NumAppPDE}
A. Quarteroni, and A. Valli.
\newblock{Numerical approximation of partial differential equations}.
\newblock{\em Springer Science \& Business Media}, Vol. 23, 2008.

\bibitem{SJOBERG}
	A. Sj{\"o}berg.
	\newblock{On the Korteweg-de Vries equation: existence and uniqueness}.
	\newblock{\em Journal of Mathematical Analysis and Applications}, 29 (1970), no.3, 569--579.
	
\bibitem{StabOnProj}
O. Steinbach.
\newblock{On the stability of the $L^2$ projection in fractional Sobolev spaces}.
\newblock{\em Numerische Mathematik}, 88 (2001), no. 2, 367--379.
	
\bibitem{TaoBO}
	T. Tao.
	\newblock{Global well-posedness of the Benjamin–Ono equation in $H^1(\mathbb{R})$}.
	\newblock{\em Journal of Hyperbolic Differential Equations}, 1 (2004),  no. 1, 27--49.
	
\bibitem{Vmurty}
	V. Thom{\'e}e, and A. S. Vasudeva Murthy.
	\newblock{A numerical method for the Benjamin–Ono equation}.
	\newblock{\em BIT Numerical Mathematics}, 38 (1998), 597--611.
\end{thebibliography}
\end{document}